\documentclass[11pt,a4paper,fleqn]{article}

\usepackage[margin=1in]{geometry}
\usepackage{graphicx}
\usepackage{float}
\usepackage{amssymb}
\usepackage{amsmath}
\usepackage{amsthm}
\usepackage{xcolor}
\usepackage{tikz}
\usepackage[ruled]{algorithm2e}
\usepackage{dsfont}
\usepackage[mathscr]{euscript}
\usepackage{longtable}
\usepackage{tabularx}
\usepackage{paralist}
\usepackage{booktabs}
\usepackage{pdflscape}
\usepackage{caption}
\usepackage{authblk}
\usepackage[authoryear]{natbib}
\usepackage[colorlinks,citecolor=blue,urlcolor=blue,linkcolor=blue]{hyperref}

\usetikzlibrary{fit}
\usetikzlibrary{shapes.geometric}
\usetikzlibrary{positioning}
\usetikzlibrary{calc}
\usetikzlibrary{arrows.meta}
\usetikzlibrary{backgrounds}

\newcommand{\ourparagraph}[1]{
  \medskip

  \noindent\emph{#1}.
}

\newcommand{\cycle}{C}
\newcommand{\fa}{\text{for all }}

\newcommand{\cev}[1]{\reflectbox{\ensuremath{\vec{\,\,\reflectbox{\ensuremath{#1}}}}}\!\!}
\newcommand{\lbl}{\ell}
\newcommand{\flbl}{\vec{\ell}}
\newcommand{\blbl}{\cev{\ell}}
\newtheorem{theorem}{Theorem}
\newtheorem{lemma}{Lemma}
\newtheorem{example}{Example}
\let\cite\citep

\title{A Numerically-safe Branch-Price-and-Cut Algorithm for the Length-Constrained Cycle Partition Problem}

\author[1]{Mohammed Ghannam\thanks{ghannam@zib.de}}
\author[1,2]{Ambros Gleixner\thanks{gleixner@htw-berlin.de}}
\author[3]{Edward Lam\thanks{edward.lam@monash.edu}}
\author[1]{Gioni Mexi\thanks{mexi@zib.de}}
\affil[1]{Zuse Institute Berlin, Germany}
\affil[2]{HTW Berlin, Germany}
\affil[3]{Monash University, Australia}

\date{}

\begin{document}
\renewcommand{\floatpagefraction}{.8}
\renewcommand{\textfraction}{.1}
\renewcommand{\topfraction}{.9}
\renewcommand{\bottomfraction}{.9}

\maketitle

\begin{abstract}
    The length-constrained cycle partition problem (LCCP) is
    a graph optimization problem in which a set of nodes must be partitioned into a minimum number of cycles.
    Every node is associated with a critical time and the length of every cycle must not exceed the critical time of any node in the cycle.
    We formulate LCCP as a set partitioning model and solve it using an exact branch-price-and-cut approach.
    Our dynamic programming-based pricing algorithm to generate improving cycles
    exploits the particular structure of the pricing problem for efficient bidirectional search and symmetry breaking.
    Computational results show that the LP relaxation of the set partitioning model produces very strong dual bounds and our branch-price-and-cut method improves significantly over the state of the art.
    It is able to solve previously solved instances in a fraction of the time and closes $14$~previously unsolved instances with numerically safe bounds,
    one of which has~$76$ nodes, a notable improvement over the previous limit of~$52$ nodes.
\end{abstract}


\section{Introduction}\label{sec:intro}
A \emph{cycle partition} of an undirected graph~$G=(V,E)$ is a partition of the set of nodes~$V$ into disjoint cycles. In the \emph{length-constrained cycle partition problem} (LCCP) introduced in~\cite{hoppmann2020minimum},
we are additionally given a \emph{critical time}~$q_i > 0$ for every node~$i \in V$ and a \emph{travel time}~$t_{i,j} \geq 0$ for every edge~$\{i,j\} \in E$.
We call a cycle $\cycle=(i_0,i_1,\ldots,i_K=i_0)$ \emph{length-feasible} if it satisfies the
length constraint
\begin{equation}\label{eq:lc}
t(\cycle) := t_{i_0,i_1} + t_{i_1,i_2} + \ldots + t_{i_{K-1},i_K}
\leq q(\cycle) := \min \{ q_{i_0},q_{i_1},\ldots,q_{i_K} \},
\end{equation}
i.e., the total time to traverse the edges of the cycle must not exceed the critical time
of any node in the cycle.
In other words, an agent that continuously travels along the cycle will visit each node~$i$ of the cycle with frequency at least~$1/q_i$.
The LCCP then amounts to computing the smallest number of disjoint, length-feasible cycles
that cover all nodes.

\begin{center}
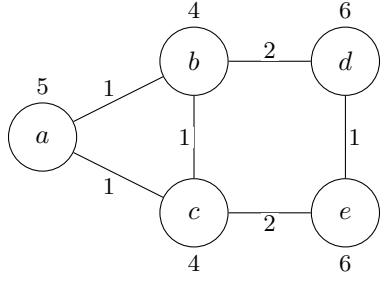

    \begin{tikzpicture}[
        node distance=2cm,
        every node/.style={draw, circle, minimum size=0.9cm, font=\small},
        edge label/.style={draw=none, rectangle, minimum size=0, font=\footnotesize, fill=white, inner sep=1pt},
        crit label/.style={draw=none, rectangle, font=\footnotesize}
    ]
        \node (A) at (0,0) {$a$};
        \node[crit label, above=-7pt of A] {5};

        \node (B) at (2,1) {$b$};
        \node[crit label, above=-7pt of B] {4};

        \node (C) at (2,-1) {$c$};
        \node[crit label, below=-7pt of C] {4};

        \node (D) at (4,1) {$d$};
        \node[crit label, above=-7pt of D] {6};

        \node (E) at (4,-1) {$e$};
        \node[crit label, below=-7pt of E] {6};

        \draw (A) -- node[edge label, above left] {1} (B);
        \draw (A) -- node[edge label, below left] {1} (C);
        \draw (B) -- node[edge label, left] {1} (C);
        \draw (B) -- node[edge label, above] {2} (D);
        \draw (C) -- node[edge label, below] {2} (E);
        \draw (D) -- node[edge label, right] {1} (E);
    \end{tikzpicture}
    \captionof{figure}{Example LCCP instance with 5 nodes. Edge labels indicate travel times $t_{ij}$; node labels show critical times $q_i$.}
    \label{fig:example-instance}
\end{center}

\begin{example}[LCCP instance]\label{ex:lccp-instance}
Consider the instance in \autoref{fig:example-instance}.
Nodes $b$ and $c$ have the tightest critical times ($q_b = q_c = 4$), limiting which cycles can include them.
The only Hamiltonian cycle $C = (a, b, d, e, c, a)$ visiting all nodes has $t(C) = 1 + 2 + 1 + 2 + 1 = 7$ and $q(C) = \min\{5, 4, 6, 6, 4\} = 4$.
Since $t(C) > q(C)$, this cycle violates the length constraint, hence at least two cycles are needed.
An optimal partition uses:
\begin{itemize}
    \item $C_1 = (a, b, c, a)$: $t(C_1) = 1 + 1 + 1 = 3 \leq 4 = q(C_1)$
    \item $C_2 = (d, e, d)$: $t(C_2) = 1 + 1 = 2 \leq 6 = q(C_2)$
\end{itemize}
\end{example}

The setting of LCCP is found to be useful, e.g., in applications where the objective is to find the smallest number of agents to conduct periodic surveillance or maintenance actions~\cite{hoppmann2022mathematical}.
Similar types of length constraints can be found in the literature on minimizing cycle length in kidney exchange programmes~\cite{Lam:2019ab}, where the length is simply the number of nodes in the cycle.
\citet{hoppmann2020minimum} show that the LCCP is NP-hard by reduction from the \emph{traveling salesman problem} (TSP), and that no polynomial-time approximation algorithm exists.
They present a compact mixed-integer programming (MIP) formulation based on the Miller, Tucker, and Zemlin (MTZ)~\cite{miller1960integer} subtour-elimination constraints for the TSP. This compact model is able to solve instances with up to 29~nodes to proven optimality.

In a later work~\citep{lccp}, the same authors introduced a MIP formulation based on exponentially many subtour elimination constraints (SEC), separated as cutting planes, together with valid inequalities derived from cliques in conflict hypergraphs, and an efficient primal heuristic.
A branch-and-cut implementation of the SEC model based on a state-of-the-art commercial
MIP solver was able to solve instances with up to 52 nodes.
 Both the MTZ and SEC models require additional auxiliary variables to count the number of used cycles in the objective function and to ensure that the total time of a cycle respects the critical time constraint of its nodes.
Despite the mentioned improvements, the linear relaxations of these compact formulations remain weak.
The formulations are moreover highly symmetric, which can hamper branching, although modern MIP solvers include dedicated symmetry-handling techniques that mitigate this to some extent.

The goal of our work is to overcome computational bottlenecks of the previous approaches by considering a set partitioning formulation based on cycle variables.
Let~$\Omega$ denote the set of all length-feasible cycles, and let~$a_i^\cycle \in \{0,1\} $ be a
constant equal to $1$ iff node~$i$ is contained in cycle~$\cycle$.
Introducing binary variables $\lambda_\cycle$ to indicate whether cycle $\cycle \in \Omega$ is used in the solution, we arrive at the set partitioning (SP) formulation
\begin{equation}\label{prob:imp}
\begin{array}{llll}
\min\;\; & \sum_{\cycle \in \Omega} \lambda_\cycle & \\\\
\text{s.t.} & \sum_{\cycle \in \Omega} a_i^\cycle  \lambda_\cycle &= 1  & \quad \fa i \in V, \\\\
              & \mbox{}\hfill\lambda_{\cycle} &\in \{0,1\} & \quad \fa \cycle \in \Omega.
       \end{array}
        \tag{SP}
\end{equation}
This reformulation of LCCP is inspired by formulations for the
cardinality-constrained multi-cycle problem in kidney exchange~\cite{Lam:2019ab}, where a maximum-value packing of cycles is required, and by exact methods in many logistics applications such as vehicle routing, see, e.g.,~\cite{costa_exact_2019}.

As in these applications, it is not viable to solve~\eqref{prob:imp} explicitly due to the
exponential number of variables.
Instead, we apply \emph{column generation} in order to solve the linear programming (LP)
relaxation of~\eqref{prob:imp} dynamically, and integrate this into an
exact \emph{branch-and-price} algorithm in order to solve LCCP instances to proven
optimality.
We refer to~\cite{cg-book,Wolsey2020,uchoa2024columngen,desrosiers2026branchprice} for a general overview of this methodology.

A major computational task in this approach is the repeated solution of the
so-called pricing problem for dynamically adding
variables with negative reduced cost to the LP relaxation of~\eqref{prob:imp}.
Given travel times $t_{i,j} \geq 0$ for $\{i,j\}\in E$ and both critical times $q_i > 0$ and
weights $\pi_i\in\mathbb{R}$ for~$i\in V$, the pricing problem amounts to computing
a cycle $C$ that satisfies the length constraint~\eqref{eq:lc} and
maximizes the sum of node weights.
We call this the \emph{length-constrained prize-collecting cycle
problem} (LCPCCP).
Note that problem data may be non-metric, i.e., travel times $t_{i,j}$ may violate the triangle inequality, and that the node weights~$\pi_i$, which are
derived as dual multipliers associated with the partitioning constraint of the node in the
LP relaxation, are not restricted in sign.
Even in the metric case and when critical times are assumed to be constant, LCPCCP is strongly NP-hard by reduction from the Hamiltonian cycle problem (\autoref{thm:nphard}).
The paper is organized as follows.

In \autoref{sec:labeling} we present a new label-setting dynamic-programming algorithm
for LCPCCP including dominance rules for reducing the search space and an efficient bidirectional search that allows us to enumerate cycles
only half-way.
In \autoref{sec:trieq}, we describe how to exploit the triangle inequality for locally metric input data, including per-node detection, pruning of zero-dual nodes, and preemptive feasibility checks.
In \autoref{sec:cuts}, we adapt two families of cutting planes from the literature, subset row cuts and clique cuts, to the set partitioning formulation~\eqref{prob:imp} of LCCP and discuss their management during column generation.
In \autoref{sec:numerical} we describe how we obtain numerically safe lower bounds without floating-point inaccuracies,
and in \autoref{sec:bnp} we integrate the resulting column generation scheme into a numerically safe
branch-price-and-cut algorithm enhanced by symmetry breaking, heuristic pricing, parallel
pricing, and early branching.
In \autoref{sec:results} we describe the results of our computational study to analyze the
performance of the new approach on benchmark instances from the literature and to quantify
the impact of the individual improvement techniques.

Our branch-price-and-cut algorithm is able to solve 52~instances with up to 76~nodes to proven optimality, closing $14$~previously unsolved instances.
Already its branch-and-price core is on average 14.7~times faster than the best previous approach (see \autoref{tab:overall}), while in addition ensuring numerically safe bounds.
Among the algorithmic contributions, the ingredients that exploit the specific structure of LCCP are, to the best of our knowledge, new: the symmetry of the undirected pricing problem, which lets bidirectional search extend labels in only one direction; symmetry breaking by critical time across the per-node pricing problems; and the exploitation of the triangle inequality for locally metric data.

A preliminary version of this work, introducing the set partitioning formulation and the branch-and-price algorithm, was presented at the INFORMS Optimization Society (IOS) Conference 2024~\cite{ios}. The present manuscript extends it with cutting-plane separation (subset-row and clique cuts) and a framework for numerically-safe dual bounds, together with a considerably expanded computational study.

\section{Solving the LP Relaxation by Column Generation}\label{sec:labeling}
\newcommand{\redcost}{\overline{c}}
\newcommand{\lblnodes}{\mathscr{N}}
\newcommand{\lbldef}{(\lblnodes, v, \redcost, t, q)}

Column generation is an advanced technique for solving large linear programs and is an essential subroutine in branch-and-price algorithms~\cite{cg-book,Wolsey2020}.
The method has recently been found~\citep{uchoa2026kantorovich} to date back to \citet{KantorovichZalgaller1951}, and was independently developed by \citet{GilmoreGomory1961}.

In order to solve the LP relaxation of \eqref{prob:imp} by column
generation, we first note that the upper bounds $\lambda_\cycle\leq 1$ are implied by the
partitioning constraints and can be removed. This avoids introducing a separate dual variable for each bound, which would otherwise enter the reduced cost of every column.
Without the upper bounds, the reduced cost depends only on the partitioning duals~$\pi_i$, which are unrestricted in sign.
We call the resulting LP the \emph{master problem} (MP).
Second, we replace $\Omega$ with a subset of columns $\Omega' \subset \Omega$ to form
the \emph{restricted master problem} (RMP).
Column generation then repeatedly optimizes (RMP), each time producing a primal-dual
pair of solutions~$(\lambda, \pi)$ and solving the so-called \emph{pricing problem} to
check whether cycles~$\cycle$ with negative reduced cost
\begin{equation}
  \label{prob:pp}
  \textstyle\redcost(\cycle) = 1 - \sum_{i \in V} a^\cycle_i \pi_i = 1 - \sum_{i \in C} \pi_i
\end{equation}
exist in $\Omega \setminus \Omega'$.
If yes, a subset of these are added to $\Omega'$ and (RMP) is
reoptimized; if not, then $(\lambda, \pi)$ is optimal for (MP).
We refer to~\cite{cg-book,Wolsey2020} for details on column generation.
As explained above, the pricing problem of minimizing~\eqref{prob:pp} amounts to LCPCCP, the length-constrained prize-collecting cycle problem.

\begin{theorem}\label{thm:nphard}
    LCPCCP is strongly NP-hard, even when the travel times are metric and the critical times are constant.
\end{theorem}
\begin{proof}
    We reduce from the Hamiltonian cycle problem, which is strongly NP-complete~\cite{gareyjohnson1979}.
    Given a graph $G_H=(V_H, E_H)$ with $n=|V_H|$, construct an LCPCCP instance on $G_H$ with unit travel times $t_{i,j}=1$ for all $\{i,j\}\in E_H$, constant critical times $q_i=n$ for all $i\in V_H$, and node weights $\pi_i=1$ for all $i\in V_H$.
    Every cycle $C$ then satisfies $t(C)=|C|\leq n=q(C)$ and is therefore length-feasible, and its objective value $\sum_{i\in C}\pi_i=|C|$ equals its number of nodes.
    Hence an optimal cycle contains $n$ nodes if and only if $G_H$ admits a Hamiltonian cycle.
    The construction uses only unit data, and unit travel times satisfy the triangle inequality, so LCPCCP is strongly NP-hard already in the metric, constant-critical-time case.
\end{proof}

\subsection{Exact Pricing by Dynamic Programming}

In the following, suppose we restrict LCPCCP by fixing one node $s\in V$ to be
contained in the cycle.
Then the resulting problem closely resembles a
\emph{resource-constrained shortest path problem} (RCSPP)~\cite{irnich_resource_2008}
typically solved in vehicle routing 
applications on a directed graph to 
find a path of minimum reduced cost
that starts at and returns to a given node
such that resources accumulated along arcs lie within a constant resource interval specified for each node.
Similarly, we propose to solve LCPCCP for each fixed start node $s$ using dynamic programming as
an implicit enumeration scheme over all length-feasible cycles.

\subsubsection{A Label-Setting Algorithm for LCPCCP}

To this end, we form \emph{partial cycles} (paths) and iteratively extend these paths with incident unvisited nodes until $s$ is visited again, essentially closing the cycle.
If at any point a partial cycle becomes infeasible (cannot be extended to a length-feasible cycle) or is proved to be dominated (would only lead to cycles with same or larger reduced cost), it is discarded, see~\autoref{sec:pruning}.
This process is repeated until all non-dominated length-feasible cycles are explored. If there are cycles with negative reduced cost, the ones with minimum reduced cost are returned.

Let~$P=(i_0,i_1, \dotsc, i_K)$ be a path in $G=(V,E)$ with nodes $i_0, \dotsc, i_K \in V$ and edges $\{i_0,i_1\}, \dotsc, \{i_{K-1},i_K\} \in E$.
The
path $P$ can be represented
by a so-called \emph{label}~$\lbl = \lbldef$ where
\begin{itemize}
    \item $\lblnodes = \lblnodes(\lbl) = \{i_1, \dotsc, i_K\}$ is the (unordered) set of visited nodes excluding the initial occurrence~$i_0=s$ of the start node (so a closed cycle, where $i_K=s$, includes~$s$),
    \item $v = v(\lbl) = i_K$ is the last node visited in~$P$,
    \item $\redcost = \redcost(\lbl) = 1 - \sum_{i \in \lblnodes } \pi_i$ is the reduced cost of~$P$,
    \item $t = t(\lbl) = \sum_{k=1}^K t_{i_{k-1},i_k} $ is the total travel time taken by~$P$, and
    \item $q = q(\lbl) = \min_{k=0,1, \dotsc, K} q_{i_k}$ is the minimum critical time of all nodes in~$P$.
\end{itemize}
The initial label for a starting node~$s$ is given as~$\lbl_s=( \{\}, s, 1, 0, q_s)$.
During the search, a label~$\lbl = \lbldef$ can be \emph{extended} using any
edge~$\{v(\lbl),j\}$ from the last node to a new node $j \in V \setminus \lblnodes$.
This creates a new label
$\lbl^+$ where $\lblnodes(\lbl^+) = \lblnodes(\lbl) \cup \{j\}$, $v(\lbl^+) = j$, $\redcost(\lbl^+) = \redcost(\lbl) - \pi_j$, $t(\lbl^+) = t(\lbl) + t_{v(\lbl),j}$, and $q(\lbl^+) = \min\{q(\lbl), q_j\}$.
We call a label~$\lbl^* \neq \lbl_s$  \emph{fully-extended} if it represents a
cycle, i.e., if $v(\lbl^*) = s$.
Note that LCCP allows singleton cycles (loops) as part of the solution; these are included
here as label~$(\{s\}, s, 1-\pi_s, 0, q_s)$.
In order to recover the cycle corresponding to a fully-extended label, we also keep track
of each label's predecessor, i.e., the label from which it was extended.

\begin{example}[Label extension and pruning]\label{ex:labeling}
We illustrate the labeling algorithm on the instance from Example~\ref{ex:lccp-instance}, pricing from start node $s = a$ with dual values $\pi = (0.5, 0.4, 0.4, 0.6, 0.6)$ for nodes $(a, b, c, d, e)$.
The initial label is $\lbl_0 = (\emptyset, a, 1.0, 0, 5)$.
Extending to node~$b$ creates
$\lbl_1 = (\{b\}, b, 0.6, 1, 4)$,
where the reduced cost decreases by $\pi_b = 0.4$, time increases by $t_{a,b} = 1$, and critical time becomes $\min\{5, 4\} = 4$.

\emph{Feasibility pruning:}
From $\lbl_1$, extending to node~$d$ gives $\lbl_2 = (\{b,d\}, d, 0.0, 3, 4)$.
Can $\lbl_2$ extend to node~$e$?
This would give $t = 3 + 1 = 4 \leq \min\{4, 6\} = 4$, so the extension is feasible.
However, if we tried to extend further to node~$c$, we would need $t > 4$, violating the critical time constraint, so such labels are pruned.

\emph{Cycle completion:}
From $\lbl_1 = (\{b\}, b, 0.6, 1, 4)$, extending to node~$c$ gives $\lbl_3 = (\{b,c\}, c, 0.2, 2, 4)$.
Closing the cycle back to node~$a$, the reduced cost becomes $0.2 - \pi_a = 0.2 - 0.5 = -0.3 < 0$.
Thus, cycle $(a,b,c,a)$ has negative reduced cost and can be added to (RMP).
\end{example}

\subsubsection{Pruning the Search Space by Feasibility and Dominance}
\label{sec:pruning}

Besides the data structures used to represent and extend labels, the efficiency of the
resulting dynamic programming method largely depends on how much of the search space needs
to be explored in order to terminate with a provably optimal cycle.
A label~$\lbl$ can be discarded if it can be proven to only lead to length-infeasible
cycles.
This is the case if~$ t(\lbl) > q(\lbl)$.
Then for any extension~$\lbl^+$ of label~$\lbl$, we have $  t(\lbl^+) \geq t(\lbl) > q(\lbl) \geq q(\lbl^+)$.
By induction, all fully-extended labels~$\lbl^*$ obtained by successive extensions of label~$\lbl$ satisfy~$t(\lbl^*) > q(\lbl^*)$, hence the corresponding cycles violate the length constraint~\eqref{eq:lc}.

An effective technique for pruning the search space further is to exploit dominance relations.
A label~$\lbl_a$ is said to \emph{dominate} label~$\lbl_b$ if they have the same end node~$v=v(\lbl_a)=v(\lbl_b)$, and all cycles reachable from successive extensions of~$\lbl_a$ have an equal or lower reduced cost than those reachable from successive extensions of~$\lbl_b$.
Then again, by induction, we may discard label~$\lbl_b$ in the search, because at least
one minimum reduced cost cycle will remain.
Formally, this yields

\begin{lemma}\label{lemma:dominate}
    A label~$\lbl_b$ can be pruned by dominance if there exists a label~$\lbl_a$ with ~$v(\lbl_a) = v(\lbl_b)$ and
    \begin{subequations}
        \begin{align}
            \redcost(\lbl_a) &\leq \redcost(\lbl_b), \label{dom-redcost} \\
            t(\lbl_a) &\leq t(\lbl_b), \text{ and} \label{dom-time} \\
            \lblnodes(\lbl_a) &\subseteq \lblnodes(\lbl_b) \label{dom-subset}.
        \end{align}
    \end{subequations}
\end{lemma}

\begin{proof}
    From~\eqref{dom-subset} we have that the set of possible extensions of label~$\lbl_b$ is a subset of the set of possible extensions of label~$\lbl_a$.
    From~\eqref{dom-subset} it also follows that $q(\lbl_a) = \min_{i\in\lblnodes(\lbl_a)} q_i \geq \min_{i\in\lblnodes(\lbl_b)} q_i = q(\lbl_b)$.
    Together with~\eqref{dom-time} this proves that any feasible extension of~$\lbl_b$ is also feasible for~$\lbl_a$. Therefore, any cycle reachable from~$\lbl_b$ is also reachable from~$\lbl_a$, and from \eqref{dom-redcost} it has an equal or lower reduced cost.
\end{proof}

\begin{example}[Dominance]\label{ex:dominance}
Continuing Example~\ref{ex:labeling}, consider two labels both ending at node~$c$:
\begin{itemize}
    \item $\lbl_a = (\{b, c\}, c, 0.2, 2, 4)$ via path $a \to b \to c$
    \item $\lbl_b = (\{c\}, c, 0.6, 1, 4)$ via path $a \to c$
\end{itemize}
Neither label dominates the other: $\lbl_a$ has the lower reduced cost ($0.2 < 0.6$), while $\lbl_b$ has the smaller visited set and the lower time. Neither set of conditions of~\autoref{lemma:dominate} holds, so both labels must be kept; indeed both can be completed to the optimal cycle $(a,b,c,a)$ of reduced cost $-0.3$, directly from~$\lbl_a$ and from~$\lbl_b$ by first extending to~$b$.

Dominance does occur once a dual vanishes. In a later iteration with $\pi_b = 0$, the two labels become $\lbl_a = (\{b,c\}, c, 0.6, 2, 4)$ and $\lbl_b = (\{c\}, c, 0.6, 1, 4)$, since visiting~$b$ no longer changes the reduced cost. Now $\lbl_b$ dominates $\lbl_a$: both end at~$c$ with the same reduced cost, but $\lbl_b$ has the lower time ($1 < 2$) and the smaller visited set ($\{c\} \subset \{b,c\}$). The detour through~$b$ can only waste time without improving any reachable cycle, so $\lbl_a$ is pruned.
\end{example}

The dominance rule can be strengthened by enriching the visited set with \emph{unreachable} nodes.
A node $u$ is unreachable from a label $\lbl$ if visiting $u$ would make it impossible to complete a feasible cycle, i.e., if $t(\lbl) + t_{v(\lbl), u} > q(\lbl)$.
When extending a label to node $j$, for each neighbor $u$ of $j$ that is unreachable from the extended label, we add $u$ to the visited set.
As these nodes cannot be part of any feasible extension, this does not affect correctness but strengthens dominance by making the visited set larger, allowing more labels to be pruned earlier in the search.

\subsubsection{Bidirectional Search}
\label{sec:bidir}

Exploring the search space from both forward and backward directions simultaneously is
another successful technique to reduce the number of explored labels.
For the RCSPP, it has been shown that
bidirectional search can be more efficient than monodirectional search~\cite{tilk_asymmetry_2017,tilk_bidirectional_2020}.
The efficiency gain comes from the fact that the
number of non-dominated labels can grow exponentially with the length of the paths.
Hence, if it is possible to explore forward and backward paths only until a halfway point and merge them to full paths, this can drastically reduce the total number of labels generated.

For LCPCCP, we can apply bidirectional search even more effectively than for the general RCSPP. Since the graph is undirected and the resource consumption by travel times is symmetric, only extensions from
one direction are needed. Specifically, in bidirectional search for a fixed start node~$s$, we only extend labels~$\lbl$
with $t(\lbl) < \frac{q(\lbl)}{2}$.
We then \emph{merge} non-dominated labels as follows.

\newcommand{\concat}{||}
Let labels~$\lbl_a$ and $\lbl_b$ both start at $s$, end at $v=v(\lbl_a) = v(\lbl_b)$, and be otherwise disjoint, i.e., $\lblnodes(\lbl_a) \cap \lblnodes(\lbl_b) = \{v\}$.
Then we create a cycle by concatenating the path represented by~$\lbl_a$ and, in reverse order, the path represented by~$\lbl_b$.
We denote the newly merged label by $\lbl_a\concat\lbl_b$, given by
$\lblnodes(\lbl_a\concat\lbl_b) = \lblnodes(\lbl_a) \cup \lblnodes(\lbl_b) \cup \{s\}$, $v(\lbl_a\concat\lbl_b) = s$, $\redcost(\lbl_a\concat\lbl_b) = \redcost(\lbl_a) + \redcost(\lbl_b) - \pi_s + \pi_{v(\lbl_a)} - 1$,
    $t(\lbl_a\concat\lbl_b) = t(\lbl_a) + t(\lbl_b)$, and
    $q(\lbl_a\concat\lbl_b) = \min\{q(\lbl_a), q(\lbl_b)\}$, see~\autoref{fig:merge} for an illustration.
The reduced cost $\redcost(\lbl_a\concat\lbl_b)$ is exactly the reduced cost $1 - \sum_{i \in \lblnodes(\lbl_a\concat\lbl_b)} \pi_i$ of the merged cycle, where the three correction terms applied to the naive sum $\redcost(\lbl_a) + \redcost(\lbl_b)$ account for the conventions in the label definition above. Each label's reduced cost carries the constant~$1$ (the cost of one cycle), so the duplicate is removed by~$-1$. The shared end node~$v(\lbl_a)=v(\lbl_b)$ lies in both~$\lblnodes(\lbl_a)$ and~$\lblnodes(\lbl_b)$, so its dual is subtracted twice and added back once by~$+\pi_{v(\lbl_a)}$. Finally, the start node~$s$ belongs to the closed merged cycle but, by convention, to neither~$\lblnodes(\lbl_a)$ nor~$\lblnodes(\lbl_b)$; its dual is therefore included by~$-\pi_s$.
We only merge a disjoint pair whose resulting cycle is length-feasible, i.e., for which $t(\lbl_a) + t(\lbl_b) \leq \min\{q(\lbl_a), q(\lbl_b)\}$.

\begin{example}[Bidirectional search]\label{ex:bidir}
Consider pricing from $s = a$ in Example~\ref{ex:lccp-instance} with $\pi = (0.5, 0.4, 0.4, 0.6, 0.6)$.
In bidirectional mode, we extend labels only while $t(\lbl) < q(\lbl)/2$.
Starting from $\lbl_0 = (\emptyset, a, 1.0, 0, 5)$, the labels are created in the order
\begin{align*}
\lbl_0 &\to \lbl_1 = (\{b\}, b, 0.6, 1, 4): && t = 1 < q/2 = 2 \text{, continue} \\
\lbl_0 &\to \lbl_2 = (\{c\}, c, 0.6, 1, 4): && t = 1 < q/2 = 2 \text{, continue} \\
\lbl_1 &\to \lbl_3 = (\{b,c\}, c, 0.2, 2, 4): && t = 2 \geq q/2 = 2 \text{, stop extending} \\
\lbl_2 &\to \lbl_4 = (\{b,c\}, b, 0.2, 2, 4): && t = 2 \geq q/2 = 2 \text{, stop extending}
\end{align*}
Labels $\lbl_3$ and $\lbl_2$ both end at node~$c$ and are
disjoint except for the last node:
$\lblnodes(\lbl_3) \cap \lblnodes(\lbl_2) = \{b,c\} \cap \{c\} = \{c\}$.

Merging them produces cycle $(a, b, c, a)$ with $t = 2 + 1 = 3 \leq 4 = q$.
The reduced cost is $\redcost(\lbl_3) + \redcost(\lbl_2) - \pi_a + \pi_c - 1 = 0.2 + 0.6 - 0.5 + 0.4 - 1 = -0.3$.
No length-feasible cycle has a smaller reduced cost, so $(a, b, c, a)$ is an optimal solution to this pricing problem.
\end{example}

\begin{figure}[!htbp]
    \centering
    \begin{tikzpicture}[every node/.style={draw, circle, minimum size=0.7cm, font=\small}]
        \node[label=180:$\lbl_3$] (a1) at (0,0) {$a$};
        \node (b1) at (1.3,0) {$b$};
        \node (c1) at (2.6,0) {$c$};
        \draw[->, black] (a1) -- (b1);
        \draw[->, black] (b1) -- (c1);

        \node[label=0:$\lbl_2$] (a2) at (5.6,0) {$a$};
        \node (c2) at (4.3,0) {$c$};
        \draw[->, black] (a2) -- (c2);

        \node[label=180:$\lbl_3 \concat \lbl_2$] (a3) at (0,-1.4) {$a$};
        \node (b3) at (1.3,-1.4) {$b$};
        \node (c3) at (2.6,-1.4) {$c$};
        \node (a4) at (3.9,-1.4) {$a$};
        \draw[->, black] (a3) -- (b3);
        \draw[->, black] (b3) -- (c3);
        \draw[->, black, thick] (c3) -- (a4);
    \end{tikzpicture}
\caption{Merging the two partial cycles $\lbl_3$ and $\lbl_2$ from Example~\ref{ex:bidir}, which share only their end node~$c$, into the cycle $(a,b,c,a)$. The thick edge is contributed by~$\lbl_2$ traversed in reverse.}
\label{fig:merge}
\end{figure}
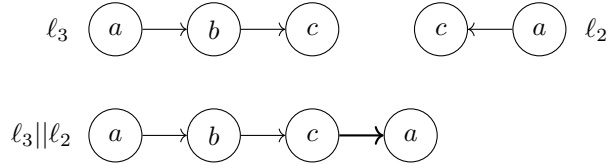

We proceed to prove formally that considering the pairs of labels described above for merging is sufficient to obtain an optimal solution of LCPCCP, even in the presence of dominance checks.

\begin{theorem}\label{lemma:merge}
   Dynamic programming with dominance and bidirectional search finds two labels~$\lbl_a$ and~$\lbl_b$ such that $\lbl_a\concat\lbl_b$ represents a length-feasible cycle with minimum reduced cost.
\end{theorem}
\begin{proof}
    Let $C=(i_0,i_1,\ldots,i_K)$ be a cycle with minimum reduced cost that is found by some version of full, monodirectional forward search with dominance checks.
    Suppose $\vec{\mathscr{L}}_C = (\flbl_0, \flbl_1, \ldots, \flbl_{K} )$ is the sequence of labels leading to~$C$ in this search, then choose~$\flbl_{k}$ to be the first label in~$\mathscr{\vec{L}}_C$ where~$t(\flbl_{k}) \geq q(\flbl_{k})/2$ is satisfied.
    Since~$\flbl_{k}$ is the first label with~$t(\flbl_{k}) \geq q(\flbl_{k})/2$, its predecessor~$\flbl_{k-1}$ satisfies~$t(\flbl_{k-1}) < q(\flbl_{k-1})/2$ and is therefore extended by bidirectional search; this extension generates~$\lbl_a:=\flbl_{k}$, a terminal forward label that is itself not extended further.
    It represents the path $(i_0,i_1,\ldots,i_k)$.

Next, consider a sequence of labels that would generate the same cycle, but in reverse order, i.e., $(i_K,i_{K-1},\ldots,i_0)$.
Denote this sequence of labels by $\cev{\mathscr{L}}_C = ( \blbl_K, \blbl_{K-1}, \ldots, \blbl_{0} )$,
    then $\blbl_{k}$ represents the path $(i_K,i_{K-1},\ldots,i_k)$.
    From $t(C) = t(\flbl_{k})+t(\blbl_{k}) \leq q(C)$ and $t(\flbl_{k}) \geq q(\flbl_{k})/2$ it follows that
    \[
    t(\blbl_{k}) = t(C) - t(\flbl_{k}) \leq q(C) - \frac{q(\flbl_{k})}{2} \leq q(C) - \frac{q(C)}{2} = \frac{q(C)}{2} \leq \frac{q(\blbl_{k})}{2},
    \]
hence $\blbl_{k}$ is length-feasible and within the halfway point.
Therefore, unless it is dominated, $\blbl_{k}$ is found by bidirectional search, and we can choose $\lbl_{b} := \blbl_{k}$, since $\lbl_{a} \concat \lbl_{b}$ represents~$C$.
If $\blbl_{k}$ is discarded due to dominance, then (by transitivity of dominance) bidirectional search must find another label that dominates $\blbl_{k}$.
Let $\blbl_{k}$ be this dominating label.
From the conditions of \autoref{lemma:dominate} it follows that $\lbl_a$ and $\lbl_b$ can be merged and that the result $\lbl_a\concat\lbl_b$ represents a length-feasible cycle with minimum reduced cost:

Labels $\lbl_a$ and $\lbl_b$ trivially start at the same node~$s=i_0=i_K$, 
and because $v(\lbl_b) = v(\blbl_k)=i_k$, they also end at the same node $i_k$.
Due to $\lblnodes(\lbl_b) \subseteq \lblnodes(\blbl_{k})$, we have
$\lblnodes(\lbl_a) \cap \lblnodes(\lbl_b) \subseteq \lblnodes(\lbl_{a}) \cap \lblnodes(\blbl_{k}) = \{i_k\}$, i.e., they do not overlap otherwise.
Hence, they can be merged to $\lbl_a\concat\lbl_b$.
It remains to show that this merged label is length-feasible and has the same reduced cost as $C$.
This follows from the fact that $\lbl_a\concat\blbl_k$ is feasible and optimal, and $t(\lbl_b) \leq t(\blbl_{k})$, $q(\lbl_b) \geq q(\blbl_{k})$, and $\redcost(\lbl_b) \leq \redcost(\blbl_{k})$ hold by domination.

\end{proof}

Note that bidirectional search may produce more cycles than monodirectional search.
The reason is that many of the labels at the halfway point
could be dominated by labels found later in monodirectional search.
Adding these potentially sub-optimal cycles to (RMP) can at the same time aid the convergence of column generation and make re-optimizing the LP much harder.
We try to safeguard against this effect by sorting the cycles by reduced cost and returning at most a fixed number ($50$ per starting node in our implementation) of the ones with lowest reduced cost.

\subsection{Non-elementary Cycle Relaxation}
\label{ng-path}

The idea of solving a relaxation of the pricing problem in order to obtain valid lower bounds has been successfully applied in vehicle routing~\cite{ng-paths, dssr}.
The similarity of the pricing problem, RCSPP, to LCPCCP allows us to adapt this technique as follows.
We call a cycle \emph{elementary} if it visits every node at most once, and \emph{non-elementary} if it revisits one or more nodes.
We employ the so-called \emph{ng-paths} relaxation~\cite{ng-paths}, allowing some non-elementary cycles to be generated from the pricing problem.
This gives us a lower bound on the minimum reduced cost, as it also contains the set of all elementary cycles.

In order to obtain a hierarchy of non-elementary relexations that converges to the result of LCPCCP with only elementary cycles,
we define neighborhoods $\mathcal{N}(i) \subseteq V$ for each node $i \in V$, representing the ``memory'' of visited nodes maintained during label extension.
When extending from node $v$ to node $j$, we record the visit (add $j$ to the visited set) only if $j \in \mathcal{N}(v)$ or $v \in \mathcal{N}(j)$, and we forbid revisiting a node in the visited set.
Hence a subcycle cannot occur between nodes that are in each other's neighborhoods.
Conversely, nodes outside mutual neighborhoods may be revisited, allowing non-elementary cycles.
Setting $\mathcal{N}(i) = V$ for all $i$ recovers the elementary pricing algorithm.
We initialize $\mathcal{N}(i) = \{i, j^*\}$ where $j^* = \arg\min_{j \neq i} t_{ij}$ is the closest neighbor of~$i$.

In the remainder of this section, we describe how subcycles in the generated cycles are detected, how the neighborhoods are expanded to prevent these subcycles from reoccurring, how a valid lower bound is obtained even while non-elementary cycles are generated, and how the relaxation interacts with bidirectional search.

\ourparagraph{Subcycle Detection}
When pricing returns a cycle $C = (i_0, i_1, \ldots, i_K = i_0)$, we check whether it is elementary by detecting repeated nodes.
We maintain a map from each node to its first occurrence index in the cycle.
If node $i_k$ was previously seen at index $k' < k$, then $(i_{k'}, i_{k'+1}, \ldots, i_{k-1})$ forms a subcycle.
All such subcycles are collected for neighborhood expansion.

\ourparagraph{Neighborhood Expansion}
For each detected subcycle $S = (s_1, s_2, \ldots, s_m)$, we expand the neighborhoods of all involved nodes to include each other:
\[
\mathcal{N}(s_j) \leftarrow \mathcal{N}(s_j) \cup \{s_1, s_2, \ldots, s_m\} \quad \text{for all } j = 1, \ldots, m.
\]
This ensures that the same subcycle cannot reoccur: since all nodes in $S$ are now mutually in each other's neighborhoods, any future visit among them will be remembered and a revisit will be blocked.

After expanding neighborhoods, pricing is repeated.
This process continues until all returned cycles are elementary, at which point the minimum reduced cost cycle is guaranteed to be elementary.
The neighborhoods grow monotonically, and in the worst case converge to $\mathcal{N}(i) = V$ for all $i$, recovering the exact elementary algorithm.

\ourparagraph{Lagrangian Bound}
Importantly, even when non-elementary cycles are returned, the pricing problem provides a valid Lagrangian lower bound on the optimal LP objective.
Let $z^*_s$ denote the minimum reduced cost found for starting node~$s$ (possibly from a non-elementary cycle).
Then $z^*_{MP} \geq z^*_{RMP} + \sum_{s \in V} \min\{0, z^*_s\}$ is a valid lower bound on the master problem, since the relaxed pricing problem contains all elementary cycles.

\ourparagraph{Bidirectional Search}
The ng-relaxation is compatible with bidirectional search from \autoref{sec:bidir} after relacing the merge condition in order to account for the partial tracking of the visited sets.
In the elementary case, we require $\tilde{\lblnodes}(\lbl_a) \cap \tilde{\lblnodes}(\lbl_b) = \{v\}$, where~$v = v(\lbl_a) = v(\lbl_b)$.
However, with partially tracked visited sets, the meeting node $v$ itself may not be tracked if its incident edges lie outside the mutual neighborhoods.
We therefore relax the condition to
\[
\tilde{\lblnodes}(\lbl_a) \cap \tilde{\lblnodes}(\lbl_b) \subseteq \{v\},
\]
i.e., the two partial paths may share at most the meeting node $v$, and possibly not even that if $v$ is untracked.
If they share any tracked node other than $v$, a genuine overlap exists and the merge is rejected.
Untracked nodes may still appear in both paths, potentially creating non-elementary cycles, which are handled by the neighborhood expansion procedure.

\subsection{Greedy Pricing Heuristic}
\label{greedy-heur}

In order to avoid expensive labeling iterations, we first run a greedy heuristic to find improving cycles.
The heuristic explores the neighborhood of each cycle $C \in \Omega'$ currently in (RMP).
For each cycle $C$, we first remove all intermediate nodes with negative dual values, provided the resulting cycle remains length-feasible.
Then, for each node $i \notin C$ with positive dual $\pi_i > 0$, we attempt to insert $i$ at every position in the cycle, keeping all feasible insertions that yield negative reduced cost of the cycle.

This heuristic is run in parallel over all cycles in $\Omega'$.
The resulting improving cycles are sorted by reduced cost, and the best ones are added to (RMP).
To avoid adding too many columns in a single iteration, we stop after adding at most $|V|$ new columns.
Only if no improving cycles are found do we proceed to the full labeling algorithm.
We observed that this greedy approach is very effective at driving down the LP objective value early in the column generation process and significantly reduces the number of expensive labeling iterations required.

\section{Exploiting the Triangle Inequality}\label{sec:trieq}
The algorithm developed up to this point works for any choice of non-negative travel times.
For many real-world applications, such as in aerial vehicle routing where distances are Euclidean, we know additionally that the triangle inequality holds, i.e., that
$t_{i,j} \leq t_{i,k} + t_{k,j}$ holds for all~$\{i,j\},\{i,k\},\{j,k\} \in E$.
Under this assumption, the set partitioning formulation can be turned into a set covering formulation, which is easier to solve because the dual solution $\pi$ becomes nonnegative.
A solution for the set partitioning formulation can always be retrieved from a set covering solution by removing nodes that appear in more than one cycle from all but one cycle.
This is guaranteed to be feasible since removing a node can only decrease the total travel time and increase the minimum critical time of the cycle.

\ourparagraph{Per-node detection}
Rather than assuming the triangle inequality holds globally, we detect for each node $i \in V$ whether it can always be feasibly removed from any cycle.
Specifically, node $i$ satisfies the triangle inequality if $t_{j,k} \leq t_{i,j} + t_{i,k}$ for all pairs $j, k \in V$.
Let $V_{trieq} \subseteq V$ denote the set of such nodes.
Note that a node enters $V_{trieq}$ only when all required bypass edges exist (missing edges with travel time $\infty$ fail the test), keeping it valid in general.
For nodes in $V_{trieq}$, we use set covering constraints ($\geq 1$) in the master problem~\eqref{prob:imp}, while the remaining nodes retain set partitioning constraints ($= 1$).
This mixed formulation allows us to exploit the triangle inequality even when it holds only partially.
The effect of using covering constraints is that the dual variables of nodes in $V_{trieq}$ are restricted to be nonnegative.
This can weaken the LP bound but also aids convergence by stabilising the dual variables, preventing their oscillation in one direction.

\ourparagraph{Skipping zero-dual nodes}
By the arguments above, we can prove the following lemma showing that nodes with zero dual can be safely ignored during pricing.
\begin{lemma}\label{lemma:tri-eq}
  Let $i \in V_{trieq}$ and $\pi\in\mathbb{R}^V$ a dual solution with $\pi_i = 0$.
  Then there exists an optimal cycle~$C$ to the pricing problem such that~$i \notin C$.
\end{lemma}
\begin{proof}
  Let $C^* = (\ldots, j, i, k, \ldots)$ be an optimal cycle containing node $i$, where $j$ and $k$ are the predecessor and successor of $i$ in $C^*$.
  Construct a new cycle $C' = (\ldots, j, k, \ldots)$ by removing $i$ and connecting $j$ directly to $k$.
  Because $i \in V_{trieq}$, the triangle inequality gives $t_{j,k} \leq t_{j,i} + t_{i,k}$, so $t(C') \leq t(C^*)$.
  Moreover, removing $i$ can only increase the minimum critical time, so $q(C') \geq q(C^*)$.
  Therefore $C'$ is feasible.
  Since $\pi_i = 0$, the reduced cost is unchanged:
  \[
  \redcost(C') = 1 - \sum_{v \in C'} \pi_v = 1 - \sum_{v \in C^*} \pi_v + \pi_i = \redcost(C^*)\,.
  \]
  Thus $C'$ is also optimal, and $i \notin C'$.
\end{proof}
Based on this lemma, we reduce the search space during label extension by skipping all neighbors $j \in V_{trieq}$ with $\pi_j = 0$.

\ourparagraph{Preemptive pruning}
Under the triangle inequality, we can also prune labels more aggressively.
When extending a label to node $j$, we check whether the cycle can still be closed feasibly by computing the time to directly return{} to the start node.
If $t(\lbl) + t_{v(\lbl), j} + t_{j, s} > q(\lbl^+)$, we skip this extension, since no feasible cycle can be formed.
This preemptive check avoids exploring labels that cannot lead to feasible cycles.

\ourparagraph{Post processing for set partitioning solutions}
The benefits above, namely the nonnegative and more stable duals and the pricing reductions of \autoref{lemma:tri-eq}, all stem from solving with set covering constraints for the nodes in $V_{trieq}$. Recovering a set partitioning solution from the covering solution afterwards is essentially free.
Since a node in $V_{trieq}$ may be covered by more than one cycle, we keep only its first occurrence and remove it from the remaining cycles.
By the triangle-inequality argument given at the start of this section, each modified cycle stays feasible, and the number of cycles is unchanged; the result is therefore a valid set partitioning solution with the same objective, hence optimal for the original problem.

\section{Cutting Planes}\label{sec:cuts}
Given a fractional LP solution $\lambda^*$, a cutting plane is a valid inequality that is violated by $\lambda^*$ but satisfied by all integer feasible solutions.
Strengthening the LP relaxation with the addition of valid inequalities has been shown to be effective~\cite{jepsen, clique, costa_exact_2019}.
In this section we present two popular families of cuts for set partitioning-based column generation, namely \emph{subset row cuts} (SRCs) and \emph{clique cuts} (CCs), followed by a discussion on whether generated cuts should be respected in the pricing routine.

\subsection{Subset Row Cuts}
Subset row cuts were introduced first by \citet{jepsen}.
For a subset of nodes (rows) $R \subseteq V$ (here representing cover constraint for each node) and $k \in \{1, \dots, |R|\}$ , an SRC has the form
\begin{equation}
    \sum_{\cycle \in \Omega} \lfloor \frac{1}{k} \sum_{i \in \cycle} a_i^{\cycle} \rfloor \lambda_{\cycle} \leq \lfloor \frac{1}{k} |R| \rfloor\,.
\end{equation}
In our implementation we use $k = 2$ and $|R| = 3$.
Intuitively, the cut is a Chv\'atal--Gomory rounding of the partitioning constraints of the rows in~$R$: summing the~$|R|$ constraints and dividing by~$k$ yields a fractional right-hand side that can be rounded down.
For our choice $k=2$ and $|R|=3$, a cycle contributes to the left-hand side only if it covers at least two of the three nodes in~$R$.
Since each of these nodes is covered exactly once in an integral partition, at most one selected cycle can cover two or more of them, which is precisely what the cut enforces; a fractional solution may instead spread the coverage of~$R$ across several cycles and thereby violate it.
The separation routine iterates through all combinations of nodes and computes coefficients to find violated cuts.

To respect these generated cuts in the pricing problem, we augment each partial cycle~$\lbl$ with one extra resource per active cut.
Let~$S$ denote the set of currently active subset-row cuts, where each cut~$r \in S$ is generated by a node set~$R_r \subseteq V$.
The added resource is a vector
\[
\rho(\lbl) \in \{0, 1, \ldots, |R_r| - 1\}^{|S|}
\]
whose entry~$\rho_r(\lbl) = |R_r \cap \lblnodes(\lbl)|$ counts how many of the nodes in~$R_r$ the partial cycle has visited.
We modify the accumulated reduced cost~$\overline{c}(\lbl)${} to include the new duals by subtracting the corresponding dual~$\sigma_r$ for every~$k$ visits to the nodes in~$R_r$.
The dominance rule is extended by the additional check~$\rho_r(\lbl) \leq \rho_r(\lbl')$ for all~$r \in S$, i.e., $\lbl$ must visit the nodes of each cut at most as many times as~$\lbl'$ does in order to dominate it.

\subsection{Clique Cuts}
Clique cuts were first presented in~\cite{original-setpart} for the set partitioning problem and later used for a variant of the vehicle routing problem in~\cite{clique}, including how the dynamic programming pricing problem solver should be updated to respect the new duals from the generated cuts.
Intuitively, two cycles that share a node can never be selected together in a partition, since that node would then be covered twice; hence among any set of cycles that pairwise conflict in this way, at most one may be used.

These cuts are based on cliques found on the so-called conflict graph.
The conflict graph encodes which variables can cause infeasibility if both set to one in an integer solution.
Each variable is represented by a node in this graph and an edge exists if setting both variables to 1 results in infeasibility.
Consequently, every cycle variable is represented by a node, and an edge exists between a pair of nodes if the corresponding cycles visit any common node.
A clique $Q \subseteq \Omega$ on that graph represents a set of variables of which at most one can have value one in any feasible (integer) solution, giving the valid inequality
\begin{equation}
    \sum_{\cycle \in Q} \lambda_\cycle \leq 1\,.
\end{equation}

We separate these cuts heuristically by enumerating $4$-combinations of the cycle variables in (RMP) and keeping those that form a clique in the conflict graph, i.e., every pair of cycles~$I, J$ in the combination shares at least one node, $I \cap J \neq \emptyset$.{}
We add all cuts that separate the current solution of (RMP).

\subsection{Respecting Cuts in Pricing}

Both SRCs and CCs are \emph{non-robust}, i.e., they can remove solutions from the feasible region of the pricing problem.
\emph{Robustness} is not strictly required for guaranteeing optimality of \eqref{prob:imp},
since cuts (constraints) are not problem-defining, so the pricing problem can safely ignore them and still claim optimality~\cite{spielt}.
However, respecting these cuts can lead to faster convergence of the column generation loop, with the tradeoff of making the pricing problem more complex.
Therefore, it is not entirely obvious which approach is better.

Ignoring the generated cuts in the pricing problem allows us to generate many more cuts than would otherwise be possible before rendering the pricing problem intractable.
This is the route we take for clique cuts: although they can be respected in the pricing problem as described in~\cite{clique}, doing so increases the complexity of the pricing problem, so we add them only to~(RMP) and ignore their duals in pricing. For subset-row cuts, we additionally evaluate respecting them in the pricing problem (\autoref{sec:bpc}).

Generally, even when a cut is ignored during pricing, it still constrains the newly generated columns in the master: as each column is added to~(RMP), we \emph{lift} every active cut to it by computing the column's coefficient in that cut. This lifting recovers part of the cut's strengthening effect at essentially no additional pricing cost.

\subsection{Managing Cuts during Branch-and-Price}

To balance the benefit of cutting planes against the overhead of separation and LP re-optimization, we limit the number of cuts added.
In each separation round, we add at most $10$ violated cuts per family (SRCs and CCs), and maintain a total limit of $100$ cuts across the entire branch-and-bound tree.
When the limit is reached, we rely on SCIP's cut management to remove inactive cuts.

When branching creates child nodes, we inherit the cuts from the parent node.
This ensures that the strengthened formulation is preserved throughout the tree without re-separating the same cuts.
For newly generated columns added to (RMP), we compute and update their coefficients for all existing cuts, ensuring the cut constraints remain valid as the formulation evolves.

\section{Numerically-safe Lowerbounds}\label{sec:numerical}
Up until this point, we have assumed that (RMP) is solved exactly, i.e., without any numerical errors.
However, in practice, (RMP) is typically solved with numerical tolerances using floating-point solvers, which can potentially lead to numerical errors through the accumulation of tiny roundoff errors.
Ensuring numerical safety in LP-based branch-and-bound has been studied extensively~\cite{neumaier2004safe, eifler2022exact, eifler2024safe},
and also in the context of exact branch and price the need for numerically safe computation has been identified for capacitated vehicle routing~\cite{safe-cvrp} and bin packing~\cite{baldacci2024numerically}.
In this section we follow the proposal made in the latter works to compute numerically safe bounds by scaling the dual values to integers and solving the pricing problem in exact integer arithmetic, and discuss how the DP-based pricing routine can be adapted to avoid errors arising from floating-point arithmetic.

Let us first consider the dual of (RMP),
\begin{equation}\label{prob:drmp}\tag{DRMP}
    \begin{array}{llll}
    \max\;\; & \sum_{i \in V} \pi_i & \\[1ex]
    \text{s.t.} & \sum_{i \in \cycle} \pi_i &\leq 1  & \quad \fa \cycle \in \Omega', \\[1ex]
                & \pi_i \in \mathbb{R} && \quad \fa i \in V\,.
           \end{array}
\end{equation}
When solving (RMP) with a standard floating-point LP solver one would get a primal-dual pair of solutions $(\lambda, \pi)$ that satisfy primal and dual constraints within some tolerance.
Concretely, we obtain an approximately feasible dual solution $\pi$ of (DRMP) that satisfies
$$
\sum_{i \in \cycle} \pi_i \leq 1 + \Delta \quad \fa \cycle \in \Omega'
$$
for some $\Delta > 0$, typically between $10^{-6}$ and $10^{-9}$.
In unsafe column generation, one would conclude optimality of the master problem when no columns with negative reduced cost exist, i.e., if all columns satisfy $1 - \sum_{i \in \cycle} \pi_i \geq -\Delta$, and use the dual objective $\sum_{i \in V} \pi_i$ as a lower bound.
However, if $-\Delta \leq 1 - \sum_{i \in \cycle} \pi_i < 0$, the branch-and-price result may be a suboptimal solution because (i) the lower bound is too large and (ii) improving columns may be missed.

We first focus on the first issue and discuss how to derive safe dual bounds
given an infeasible dual solution $\pi^{inf}$.
The key idea is to repair this infeasibility by reducing the value of the left-hand side of the dual constraints.
To ensure that no inaccuracies occur due to floating-point arithmetic, we first scale up and approximate the dual values by integers.
We multiply the dual values by a large scaling factor $M$ and round down, obtaining integers
\begin{equation}\label{equ:piint}
\pi^{int} = \lfloor M \cdot \pi^{inf} \rfloor\,,
\end{equation}
then test dual feasibility exactly in integer arithmetic, i.e., whether $\sum_{i \in \cycle} \pi^{int}_i \leq M$ holds for all $\cycle \in \Omega'$.
If the solution is infeasible, we reduce the scaling factor by a factor of ten, i.e., we set $M \leftarrow M / 10$, and recompute \eqref{equ:piint}.
This is repeated until an exactly dual feasible integer solution is found, or $M < 1$, in which case the repair fails.
Because the partitioning duals~$\pi_i$ are unrestricted in sign (see~\eqref{prob:drmp}), the repair only needs to enforce the cycle constraints. When the triangle inequality is exploited via covering constraints for the nodes in~$V_{trieq}$ (\autoref{sec:trieq}), the corresponding duals must additionally be nonnegative; we therefore change any negative covering dual to zero before scaling.

With this solution, the dynamic programming procedure in \autoref{sec:labeling} can be run using integer arithmetic.
If no negative reduced cost columns are found then one can conclude that the dual bound $\sum_{i \in V} \pi^{int}_i/M$ obtained from the repaired dual solution is numerically safe.
The procedure is summarized in \autoref{alg:dual-repair}.

\begin{algorithm}[htbp]
\caption{Dual Repair for Numerically Safe Bounds}\label{alg:dual-repair}
\KwIn{Floating-point dual solution $\pi^{inf}$, generated columns $\Omega'$}
\KwOut{Integer dual solution $\pi^{int}$, scaling factor $M$}
$M \leftarrow 10^{15}$\;
\Repeat{dual feasible}{
    $\pi^{int}_i \leftarrow \lfloor M \cdot \pi^{inf}_i \rfloor$ for all $i \in V$\;
    dual feasible $\leftarrow$ \textbf{true}\;
    \ForEach{$\cycle \in \Omega'$}{
        \If{$\sum_{i \in \cycle} \pi^{int}_i > M$}{
            dual feasible $\leftarrow$ \textbf{false}\;
            \textbf{break}\;
        }
    }
    \If{not dual feasible}{
        $M \leftarrow M / 10$\;
        \lIf{$M < 1$}{\textbf{fail}: duals not feasible}
    }
}
\Return{$\pi^{int}$, $M$}
\end{algorithm}
Since this procedure progressively decreases the sum of the values of the dual solution, the safe dual bound is potentially worse than the unsafe dual bound.
However, since the objective value is always integral, this would only lead to a different bound if the safe dual solution value crosses the integer boundary, i.e., if $\lceil ( \sum_{i \in V} \pi_{i}^{int} ) / M \rceil < \lceil \sum_{i \in V} \pi_{i}^{inf} \rceil $.

It is worth mentioning that the cuts described in~\autoref{sec:cuts} are naturally numerically safe, since they are based solely on combinatorial arguments that are not affected by floating-point errors.
Computing the cut violation of the current (RMP) solution $\lambda$ in floating-point arithmetic can lead to incorrect determination whather a cut is violated.
This may have the effect to not include violated cuts or include satisfied cuts and potentially affects performance, but does not compromise the correctness of the final result.
Adding cuts leads to now dual variables in (DRMP).
The exact dual feasibility check need not be modified to account for the new cut duals.
Since the duals only appear with nonnegative coefficients in the dual constraints, and they have nonnegative values, a dual feasible solution can also be constructed from one that has positive values for the cut duals by setting these duals to zero.

Note that the difference between the safe and unsafe bound can have several effects.
On the one hand, the safe bound may be weaker than the exact dual bound. If the LP relaxation at a node is integral and the unsafe bound already certifies optimality, a safe value that rounds up to one less than this integer may fail to prove optimality and force a branching.
On the other hand, the unsafe bound may be invalidly strong. Since the dual solution is only feasible within tolerance, the unsafe value can slightly exceed the true optimum: a safe value of $7.9999$ correctly rounds up to~$8$, whereas an unsafe value of $8.0001$ would round up to~$9$, which is precisely the error that the exact repair prevents.
In our experiments neither situation occurred: on every instance where both the safe and the unsafe configuration converge, the rounded safe and unsafe bounds agree (see~\autoref{sec:numsafety}), so ensuring numerical safety never changed a certified bound.

\section{The Branch-and-Price Algorithm}\label{sec:bnp}
Finally, let us discuss different aspects of how we extend the column generation method described in \autoref{sec:labeling} into an exact branch-and-price algorithm.

\ourparagraph{Edge Branching}
By default, we use a standard branching strategy based on implicit edge variables as described in~\cite{costa_exact_2019}.
For a fractional LP solution, we define implicit edge variables $x_{ij} = \sum_{\cycle \in \Omega': \{i,j\} \in \cycle} \lambda_\cycle$ representing the total usage of edge $\{i,j\}$ across all cycles.
We branch on the most used fractional edge variable that does not belong to a singleton cycle.
The rationale is that the most used edge would create the biggest disturbance in the subproblem, leading to an early fathoming due to infeasibility or finding a feasible integer solution.

In the down branch ($x_{ij} = 0$), we add edge $\{i,j\}$ to a set of \emph{deleted edges} and fix all cycle variables to zero that contain this edge.
In the up branch ($x_{ij} = 1$), we add edge $\{i,j\}$ to a set of \emph{fixed edges} and fix all cycle variables to zero that contain node $i$ or $j$ but do not contain edge $\{i,j\}$.
The labeling algorithm must also respect these branching decisions.
For deleted edges, we simply skip any extension along a deleted edge during label extension, and reject label merges in bidirectional search if the connecting edge is deleted.
For fixed edges, we automatically extend labels along fixed edges: when a label reaches a node $v$ with a fixed edge $\{v, w\}$, we immediately extend to $w$ without considering other neighbors.
This ensures that all generated cycles respect the branching constraints while avoiding redundant exploration.

\ourparagraph{Ryan-Foster Branching}
Ryan-Foster~\cite{ryan-foster} is an alternative branching technique for set partitioning master problems that branches on pairs of nodes rather than edges.
For a pair of nodes $i,j \in V$, we compute an implicit \emph{pair} variable $p_{ij} = \sum_{\cycle \in \Omega_{ij}} \lambda_\cycle^{*}$, where $\Omega_{ij}$ is the set of cycles containing both $i$ and $j$.
Given a fractional pair variable $p_{ij}$, we branch to eliminate this fractionality.

We maintain sets $M_i^n, C_i^n$ for each node $i \in V$ at branch-and-bound node $n$, representing nodes that \emph{must} and \emph{cannot} appear together with $i$ in the same cycle, respectively.
Let $n'$ denote the parent of node $n$:
\begin{itemize}
    \item To enforce $p_{ij} = 0$ (nodes in different cycles): fix all cycles containing both $i$ and $j$ to zero, and set $C_i^n \leftarrow C_i^{n'} \cup \{j\}$, $C_j^n \leftarrow C_j^{n'} \cup \{i\}$.
    \item To enforce $p_{ij} = 1$ (nodes in same cycle): fix all cycles containing exactly one of $i$ or $j$ to zero, and set $M_i^n \leftarrow M_i^{n'} \cup \{j\}$, $M_j^n \leftarrow M_j^{n'} \cup \{i\}$.
\end{itemize}

The labeling algorithm must be modified to respect these branching constraints:
\begin{itemize}
    \item We add resources $m(\lbl)$ and $c(\lbl)$ to each label, representing nodes that must and cannot be visited, respectively.
    \item When extending label $\lbl$ to node $v$: $m(\lbl^+) = m(\lbl) \cup M_v$ and $c(\lbl^+) = c(\lbl) \cup C_v$.
    \item During extension we never visit a forbidden node, maintaining $c(\lbl) \cap \lblnodes(\lbl) = \emptyset$; the requirement that all must-nodes are visited, $m(\lbl) \subseteq \lblnodes(\lbl)$, is checked only when the cycle is closed.
    \item Dominance requires additional conditions: $\lbl'$ dominates $\lbl$ only if additionally $m(\lbl') \subseteq m(\lbl)$ and $c(\lbl') \subseteq c(\lbl)$.
    \item In bidirectional search, merging labels $\lbl_a$ and $\lbl_b$ additionally requires: $c(\lbl_a) \cap \lblnodes(\lbl_b) = \emptyset$, $c(\lbl_b) \cap \lblnodes(\lbl_a) = \emptyset$, and $m(\lbl_a) \cup m(\lbl_b) \subseteq \lblnodes(\lbl_a) \cup \lblnodes(\lbl_b)$.
\end{itemize}

\ourparagraph{Node Selection and Processing}
For branch-and-bound node selection, we employ the best-estimate rule with plunging described in \cite{achterberg_constraint_2007}.
After selecting a branch-and-bound node, we solve the LP relaxation using column generation.
In each pricing iteration, the dynamic programming algorithm from \autoref{sec:labeling} is called from each starting node $s\in V$.
As these calls are independent, they are run in parallel.

\ourparagraph{Symmetry Breaking}
Since every cycle can be generated starting from any of its nodes, we apply symmetry breaking to avoid generating duplicate cycles from different pricing problems.
For each call to the dynamic program from a fixed starting node $s$, we restrict the search to only visit nodes with critical time at least as large as $q_s$.
For nodes with equal critical time, we use the node index as a tie-breaker, only allowing nodes $j$ where either $q_j > q_s$ holds or $q_j = q_s$ and $j > s$ hold.

This ensures that each cycle is generated exactly once: from the starting node with the smallest critical time (and smallest index among ties).
Moreover, ordering by critical time rather than index balances the computational effort across pricing problems.
The problems with smaller critical time~$q_s$ consider more nodes, but prune labels more agressively due to the length constraint.
The problems with larger~$q_s$ prune labels less agressively, but in turn these pricing problems admit fewer nodes to start with, namely only those with critical time at least~$q_s$.

When branching decisions fix an edge $(s, j)$ to be in the solution, we can skip the entire pricing problem starting from $s$ if $j$ has smaller critical time (or equal critical time but smaller index), since any cycle using this edge would be found by the pricing problem starting from $j$.
Similarly, when using Ryan-Foster branching, if $M_s$ contains a node with smaller critical time (or equal critical time but smaller index), the pricing problem starting from $s$ can be skipped.

\ourparagraph{Pricing}
We start each pricing call with the greedy heuristic described in~\autoref{greedy-heur}, returning any negative reduced cost columns found.
If this fails to find a cycle with negative reduced cost, then a relaxed version of the pricing problem is run with limited neighborhoods that only include the closest node and the node itself (as described in~\autoref{ng-path}).
Whenever we get non-elementary cycles, we increase the neighborhood of each node to include nodes that appear together with that node in a subcycle.
From this relaxation we always compute a Lagrangian lower bound~\cite{cg-primer}.
Let~$z^*_s$ be the optimal objective value of the pricing problem for a fixed starting node~$s$, and let~$z^*_{MP}$, $z^*_{RMP}$ be the optimal objective values of the master problem~MP, i.e., the LP relaxation of~\eqref{prob:imp}, and its restriction~RMP, respectively, then $z^*_{MP} \geq z^*_{RMP} + \sum_{s \in V} \min\{0, z^*_s\} =: LB_{lg}$.{}
From integrality of the objective function, we can use $\lceil LB_{lg} \rceil$ as a valid lower bound.
Whenever no negative reduced cost cycle is found in this relaxation, we exit pricing and declare the LP as solved to optimality.
The overall column generation procedure is summarized in \autoref{alg:cg-loop}.
To improve convergence of the column generation loop, we employ dual stabilization using smoothed duals~$\tilde{\pi} = \alpha \pi + (1-\alpha) \pi^{prev}$, where~$\pi^{prev}$ is the dual solution from the previous iteration and~$\alpha \in (0,1]$ is a smoothing parameter.

\begin{algorithm}[htbp]
\caption{Column Generation Loop}\label{alg:cg-loop}
\KwIn{RMP with initial columns $\Omega'$, upper bound $UB$}
\KwOut{Optimal LP solution or lower bound}
Initialize neighborhoods $\mathcal{N}(i) \leftarrow \{i, j^*\}$ for all $i \in V$\;
\Repeat{no negative reduced cost cycles found}{
    Solve RMP to obtain primal-dual pair $(\lambda, \pi)$\;
    \If{numerically safe}{
        $(\pi, M) \leftarrow$ \textsc{RepairDuals}$(\pi, \Omega')$\tcp*[r]{\autoref{alg:dual-repair}}
        Perform the subsequent bound computations in exact integer arithmetic\;
    }
    Run greedy heuristic on cycles in $\Omega'$\;
    \If{improving cycles found}{
        Add cycles to $\Omega'$\;
        \textbf{continue}\;
    }
    \Repeat{all cycles elementary}{
        Run relaxed pricing with neighborhoods $\mathcal{N}$\;
        Compute Lagrangian bound $LB_{lg} \leftarrow z^*_{RMP} + \sum_{s \in V} \min\{0, z^*_s\}$\;
        \lIf{$\lceil LB_{lg} \rceil \geq UB$}{\Return{$\lceil LB_{lg} \rceil$}}
        \ForEach{non-elementary cycle with subcycle $S$}{
            $\mathcal{N}(s_j) \leftarrow \mathcal{N}(s_j) \cup S$ for all $s_j \in S$\;
        }
    }
    \If{negative reduced cost cycles found}{
        Add cycles to $\Omega'$\;
    }
}
\Return{optimal LP solution $(\lambda, \pi)$}
\end{algorithm}

\ourparagraph{Early Branching}
Another acceleration idea based on integrality, inspired by the technique in~\cite{coloring}, is to skip pricing in some branch-and-bound node $n${} by looking at the lower bound $\lceil LB_{lg} \rceil$ of the parent node.
Let $z^*_{RMP}$ be the first RMP objective value for $n$ obtained after removing columns due to branching. If~$\lceil z^*_{RMP} \rceil = \lceil LB_{lg} \rceil$, then computing $z^*_{MP}$ exactly will not improve over the lower bound of the parent node. Therefore, we can skip pricing in this subproblem and perform early branching.

\ourparagraph{Farkas Pricing and RMP Initialization}
If the RMP at a subproblem becomes infeasible due to branching and removal of columns, we use Farkas pricing \cite{Nunkesser2006,CeselliGLNS08} to generate new columns that render the RMP feasible again, or to prove that the MP is infeasible.
The root node RMP is always feasible, because we initialize it with the trivially feasible singleton cycles and the cycles in the primal solution generated by the Most-Critical-Vertex-Based Heuristic (MCV) from~\cite{lccp}.

\ourparagraph{Heuristic Branch-and-Price}
The exactness of the branch-and-price algorithm relies on solving each pricing problem to proven optimality.
When relaxing this rule, we obtain a fast but inexact variant, which we use to compute strong primal solutions quickly.
One possibility to solve the pricing problem heuristically is to relax the checks for dominance during labeling, thereby possibly pruning optimal LCPCPP solutions by dominance.
Concretely, we drop the visited-set subset condition and the cut-resource conditions of \autoref{sec:cuts} required for exact dominance and say that a label~$\lbl'$ (heuristically) dominates~$\lbl$ as soon as $\redcost(\lbl') \leq \redcost(\lbl)$, $t(\lbl') \leq t(\lbl)$ and $q(\lbl') \geq q(\lbl)$.
This discards far more labels and substantially accelerates pricing; in addition, the ng neighborhood-expansion step of \autoref{ng-path} is skipped.
Because pricing is no longer exact, the minimum reduced cost may be missed and the resulting Lagrangian bound is not valid.
However, it can act as a strong primal heuristic to initialize an exact branch-and-price solve.

\section{Computational Results}\label{sec:results}
We implemented the branch-and-price algorithm described in \autoref{sec:bnp} using the
Rust wrapper of the branch-and-price framework
SCIP~\cite{russcip,BestuzhevaBesanconEtal2023_Enabling,scip9}.
The dynamic programming method from \autoref{sec:labeling} is called from each starting node in parallel using the Rayon library~\cite{rayon_rs}.
The rest of SCIP is sequential.
All experiments were run with a time limit of two hours on a cluster of identical machines, each equipped
with two Intel(R) Xeon(R) Gold 6338 processors running at 2.0\,GHz (up to 3.2\,GHz), with 32~cores per socket for a total of 64~cores and 128~threads per node, and 1\,TB of RAM.
Each run had exclusive access to a node, and both our method and the \textsc{bnc-sec} baseline used up to 16~threads: our pricing through explicit parallelization and Gurobi through its default threading.
We use the standard benchmark set from~\cite{lccp}, which consists of 84~instances with 14 to 100~nodes.

Our experimental study is guided by the following questions:
\begin{enumerate}
    \item How do the exact and heuristic branch-and-price variants compare against the branch-and-cut method of~\cite{lccp}, both in terms of primal solution quality and dual bounds?
    \item Which algorithmic components contribute the most to performance?
    \item What is the cost and benefit of ensuring numerical safety?
    \item Do cutting planes help to improve the branch-price-and-cut algorithm?
\end{enumerate}
In order to answer these questions, in \autoref{sec:overall}, we first evaluate overall performance, comparing the primal and dual sides of our approach against the branch-and-cut baseline.
The best solutions found by the heuristic configurations in \autoref{sec:overall} serve as initial upper bounds for all subsequent experiments, in order to focus on the dual task of proving optimality.
In \autoref{sec:ablation}, we analyze the impact of the algorithmic components, isolating the contribution of each through an ablation study and comparing alternative design choices.
Our numerically safe implementation is compared against an unsafe floating-point version in \autoref{sec:numsafety}.
Finally, in \autoref{sec:bpc} we investigate the effect of adding cutting planes.

Unless stated otherwise, all experiments use edge branching, see \autoref{sec:bnp}, and the numerically safe implementation described in \autoref{sec:numerical}.
We report the shifted geometric mean (SGM) of solving times with a shift of 1~second, including time outs at the time limit.
Within each table, the SGM is taken over the set of instances solved by at least one configuration listed in that table; consequently, absolute SGM values are not directly comparable across tables.

\subsection{Overall Performance}\label{sec:overall}
We compare the primal and dual performance of the following configurations:
\begin{itemize}
    \item \textsc{bnc-sec}: branch and cut with subtour elimination constraints from~\cite{lccp} based on Gurobi~13.0.1~\cite{gurobi}
    \item \textsc{bnp-mcv}: branch and price with bidirectional labeling (\autoref{sec:bidir}), parallelization, symmetry breaking, and early branching (\autoref{sec:bnp}), and greedy pricing (\autoref{greedy-heur}), using only the \textsc{MCV} heuristic solutions from~\cite{lccp} as initial upper bounds
    \item \textsc{heurbnp}: branch and price with the heuristic pricing variant of \autoref{sec:bnp}, which relaxes the dominance rule to trade exactness for speed
    \item \textsc{heurbnp-rf}: same as \textsc{heurbnp}, but with Ryan-Foster branching
    \item \textsc{bnp-full}: same as \textsc{bnp-mcv}, but initialized with the best known solution across all heuristic configurations as initial upper bound
    \item \textsc{bnc-full}: same as \textsc{bnc-sec}, but warm-started with the same best known solution as \textsc{bnp-full}
\end{itemize}
The starting point for all methods is the greedy \textsc{MCV} heuristic from~\cite{lccp}, which provides an initial feasible solution for each instance.
The finally best known solution (BKS) for each instance is the best primal bound found by any of the heuristic configurations (\textsc{MCV}, \textsc{heurbnp}, \textsc{heurbnp-rf}).  \textsc{bnp-full} is initialized with this BKS as the initial primal  bound.
Tab.~\ref{tab:overall} reports aggregated statistics on primal and dual performance of the five methods.

\begin{table}[!htbp]
    \centering
\small
\begin{tabular}{l rrrr rrr}
    \toprule
    & \multicolumn{4}{c}{Primal} & \multicolumn{3}{c}{Dual} \\
    \cmidrule(lr){2-5} \cmidrule(lr){6-8}
    Setting & Sol Impr & Avg Impr (\%) & T/O & Gap to BKS (\%) & Solved & Gap (\%) & Time [s] \\
    \midrule
    \textsc{bnc-sec}     & 23 & 13.5 & 47 &  5.8 & 37 & 30.3 & 147.2 \\
    \textsc{bnp-mcv}     & 26 & 14.0 & 38 &  4.9 & 46 &  0.4 & 19.8 \\
    \textsc{heurbnp}     & 41 & 13.7 & 22 &  2.1 & --- & --- & 6.7 \\
    \textsc{heurbnp-rf}  & \textbf{47} & \textbf{14.4} & \textbf{14} & \textbf{0.2} & --- & --- & 5.6 \\
    \midrule
    \textsc{bnc-full}    & --- & --- & --- & --- & 40 & 27.8 & 109.4 \\
    \textsc{bnp-full}    & --- & --- & --- & --- & \textbf{51} & \textbf{0.4} & \textbf{10.0} \\
    \bottomrule
\end{tabular}

    \caption{Overall performance. Left: primal performance relative to the greedy \textsc{MCV} heuristic baseline. Right: dual performance. Sol Impr: instances with a better primal bound than the greedy baseline; Avg Impr: average percentage improvement on improved instances; T/O: timeouts; Gap to BKS: average gap to the best known solution; Solved: instances solved to optimality; Gap: average optimality gap; Time: shifted geometric mean of solving times (shift of 1s) over the instances solved to optimality by at least one of the exact configurations.}
    \label{tab:overall}
\end{table}

\textsc{bnc-sec} is able to improve the primal baseline given by the solution of the greedy \textsc{MCV} heuristic on 23~instances and \textsc{bnp-mcv} on 26~instances.
In comparison, the heuristic branch-and-price method \textsc{heurbnp} proves particularly effective: by using a relaxed pricing algorithm that trades exactness for speed, it improves solutions on 41~instances with an average gap to the best known solution of only 2.1\%.
Adding Ryan-Foster branching (\textsc{heurbnp-rf}) improves solutions on 47~instances, six more than \textsc{heurbnp}, and lowers the average gap to the best known solution to 0.2\%.
Both heuristic variants are moreover fast: their shifted geometric mean running times over the instances solved by an exact configuration are 6.7~seconds (\textsc{heurbnp}) and 5.6~seconds (\textsc{heurbnp-rf}), below those of the exact methods, since they forgo proving optimality.
Notably, neither heuristic dominates the other: \textsc{heurbnp} finds a strictly better solution than all other configurations on 1~instance, while \textsc{heurbnp-rf} does so on 8~instances, which motivates initializing \textsc{bnp-full} with the best solution found across all heuristic configurations.

On the dual side, \textsc{bnp-mcv} solves 46~instances to optimality compared to 37~for \textsc{bnc-sec}.
This amounts to nine additional instances, and the branch-and-price variant \textsc{bnp-mcv} is moreover over 7~times faster than \textsc{bnc-sec} in terms of shifted geometric mean.
Since both are initialized with the same \textsc{MCV} solution, this is a fair comparison.
It is moreover conservative with respect to the solver: our branch-and-price runs on the academic solver SCIP, whereas \textsc{bnc-sec} uses the commercial solver Gurobi, which displays substantially faster performance on general mixed-integer programs.

The relaxed pricing in \textsc{heurbnp} also achieves the same (rounded) root dual bound as exact column generation on every instance where the latter converges.
However, note that this heuristic ``dual bound'' comes with no formal guarantee of exactness.
When the best known solutions from the heuristic configurations are fed back as initial upper bounds, \textsc{bnp-full} solves 51~instances to optimality, 14~more than \textsc{bnc-sec}, and is 14.7~times faster.
To control for the stronger initial upper bound available to \textsc{bnp-full}, we also run the branch-and-cut baseline warm-started with the same best known solution, denoted \textsc{bnc-full}. Even with this advantage, \textsc{bnc-full} solves 40~instances, still 11~fewer than \textsc{bnp-full}, and remains over 10~times slower in shifted geometric mean.
The largest instance solved by \textsc{bnp-full} features 76~nodes compared to 52~nodes for the branch-and-cut baseline.

To summarize, both primal and dual performance are significantly improved using branch and price.
In the following experiments, we use \textsc{bnp-full} as the baseline configuration.

\subsection{Analysis of Algorithmic Components}\label{sec:ablation}
Our best version \textsc{bnp-full} combines the following techniques: bidirectional labeling (\autoref{sec:bidir}), parallelization, symmetry breaking, and early branching (\autoref{sec:bnp}), greedy pricing (\autoref{greedy-heur}), and improved initial upper bounds from the heuristic branch-and-price variants (\autoref{sec:bnp}).{}
In order to quantify the performance impact of each component, we compare \textsc{bnp-full} against modified versions where one technique is disabled at a time:
\begin{itemize}
    \item \textsc{nosymbr}: without symmetry breaking
    \item \textsc{nobidir}: with mono-directional labeling
    \item \textsc{nopar}: without parallel pricing
    \item \textsc{noearly}: without early branching
    \item \textsc{nogreedy}: without greedy pricing
    \item \textsc{nobestsol}: without improved initial upper bounds
    \item \textsc{withtrieq}: exploiting triangle inequality (\autoref{sec:trieq})
\end{itemize}
Tab.~\ref{tab:ablation} reports the aggregated results including statistics on the root performance and reports the ``virtual best'' performance over all methods that could be achieved if we could choose the best method for each instance with perfect foresight.
In the following, we first examine the quality of the root LP relaxation, which explains the overall behavior of the algorithm, then quantify the impact of the individual acceleration techniques, before taking a closer look at the pricing process and at three alternative design choices.

\begin{table}[!htbp]
    \centering
\small
\begin{tabular}{lrrrrrr}
    \toprule
    Setting & Solved & At Root & Root LP & Time [s] & Ratio & Gap (\%) \\
    \midrule
    \textsc{bnp-full}  & \textbf{51} & \textbf{49} & \textbf{52} & \textbf{7.6} & \textbf{1.00} & \textbf{0.2} \\
    \midrule
    \textsc{nosymbr}   & 46 & 44 & 46 & 27.3 & 3.62 & 0.2 \\
    \textsc{nobidir}   & 47 & 46 & 47 & 24.2 & 3.20 & 0.2 \\
    \textsc{nopar}     & 49 & 47 & 49 & 16.3 & 2.16 & 0.2 \\
    \textsc{noearly}   & 51 & 49 & 52 &  7.4 & 0.98 & 0.2 \\
    \textsc{nogreedy}  & 51 & 49 & 52 &  8.5 & 1.13 & 1.9 \\
    \textsc{nobestsol} & 46 & 35 & 51 & 15.6 & 2.07 & 1.4 \\
    \textsc{withtrieq} & 49 & 47 & 52 &  9.7 & 1.28 & 0.7 \\
    \midrule
    Virtual Best       & 52 & --- & --- &  6.1 & 0.81 & --- \\
    \bottomrule
\end{tabular}

    \caption{Ablation study. Solved: instances solved to optimality; At Root: solved at root node; Root LP: root LP relaxation converged; Time: shifted geometric mean of solving times (shift of 1s) over instances solved by at least one configuration; Ratio: time relative to \textsc{bnp-full}; Gap: average root gap over the instances solved by at least one configuration.}
    \label{tab:ablation}
\end{table}

\ourparagraph{LP Relaxations}
The column generation loop of \textsc{bnp-full} converges to an optimal root LP solution for 52~instances, running into the memory or time limit in the remaining cases.
The key advantage of the set partitioning formulation is evident in the tightness of its LP relaxation: 49~instances are solved at the root node without branching.
Comparing the root dual bounds on the 52~instances where \textsc{bnp-full} converges, and measuring both against the same optimal objective, the set partitioning formulation attains an average gap of only 0.4\%, compared to 30.3\% for the SEC formulation of~\cite{lccp}.
On 44 of these instances, \textsc{bnp-full} produces a strictly better root dual bound than \textsc{bnc-sec}.
Furthermore, on all 51~instances solved to optimality, the root dual bound of the set partitioning formulation (rounded up due to integrality) is at least $z^* - 1$, demonstrating the tightness of the LP relaxation.
With most instances decided at the root, overall performance therefore hinges on how quickly the root LP relaxation is solved, i.e., on the efficiency of pricing, which is precisely what the following acceleration techniques target.

\ourparagraph{Impact of the Improvement Techniques}
Symmetry breaking proves to be the most impactful technique: disabling it results in a 3.6-fold slowdown and reduces the number of solved instances from 51 to 46.
Bidirectional labeling is the second most effective technique, with monodirectional labeling resulting in a 3.2~times slowdown and 4~fewer instances solved.
Parallelization provides a 2.2-fold speedup but only affects the number of solved instances marginally.
Interestingly, early branching has negligible impact on these instances, with nearly identical solving times and the same number of instances solved.
Disabling early branching (\textsc{noearly}) uniquely solves instance t84\_brazil58, which \textsc{bnp-full} does not solve within the time limit.
The virtual best of all ablation configurations solves 52~instances, one more than \textsc{bnp-full}.

\ourparagraph{Greedy Pricing Heuristic}
We next examine the two components that keep the individual pricing iterations cheap: the greedy pricing heuristic and the non-elementary relaxation.
The greedy pricing heuristic described in \autoref{greedy-heur} proves highly effective at finding negative reduced cost columns.
Across all 84~instances, the greedy heuristic succeeds in 93\% of pricing calls on average, contributing 77\% of all columns added during column generation.
The remaining 23\% of columns are found by the labeling algorithm, which is invoked only when the greedy heuristic fails to find improving columns.
Furthermore, the greedy heuristic reduces the number of neighborhood expansions in the ng-relaxation from an average of 7.5~per instance (without greedy) to 3.6~per instance (with greedy), indicating that greedy pricing often finds columns that would otherwise require expanding the ng-neighborhoods.
Despite generating more total columns, \textsc{bnp-full} with greedy pricing is faster than without, as the greedy heuristic avoids the more expensive labeling algorithm in most pricing iterations.
Disabling greedy pricing solves the same 51~instances with the same root bound quality but is about 13\% slower, so greedy pricing improves convergence speed rather than the bounds themselves.

\ourparagraph{Non-elementary Relaxation}
The ng-relaxation described in \autoref{ng-path} allows the pricing algorithm to generate non-elementary cycles with limited neighborhoods, expanding them only when necessary.
Across all instances, neighborhoods grow to at most 67\% of full size, with an average maximum of 38\%.
Notably, no instance requires neighborhoods to reach 100\%, demonstrating that the ng-relaxation effectively reduces pricing complexity while still finding optimal solutions.
On average, 7.5~neighborhood expansions are needed per instance when non-elementary cycles are detected.
We do not include a configuration that disables the ng-relaxation in the ablation study, since disabling it amounts to forcing full (elementary) neighborhoods from the start, precisely the worst case the relaxation avoids. As the neighborhoods never grow beyond 67\% of full size, this elementary variant would only ever be slower, so the relaxation is always beneficial.

\ourparagraph{Triangle Inequality}
Finally, we consider three alternative design choices (exploiting the triangle inequality, Ryan-Foster branching, and dual stabilization), each of which improves an isolated metric, yet fails to improve overall performance.
Exploiting the triangle inequality (\textsc{withtrieq}) reduces root column generation iterations from 126 to 64~on average, yet only 49~instances are solved.
This is because the set covering relaxation used for triangle inequality nodes can weaken the root dual bound: on 2 of the 52~instances where the root LP of both configurations converges, the dual bound is exactly 1~unit weaker. These are precisely the 2~instances that \textsc{withtrieq} fails to solve, and
\textsc{withtrieq} does not solve any additional instances compared to \textsc{bnp-full}.

\ourparagraph{Branching Strategy}
In addition to the above experiments, we tested the effect of switching the branching rule to Ryan-Foster branching instead of edge branching (\autoref{sec:bnp}).
To expose enough branching to compare the two rules, we run both strategies without the improved initial upper bounds; with these bounds almost all instances are solved at the root, leaving too few branchings to observe. For this reason Ryan-Foster branching is reported here separately and is intentionally omitted from the ablation table above.
Edge branching solves 46~instances (SGM 10.5s) while Ryan-Foster solves 45~(SGM 17.6s), making it 1.7~times slower.
Over the 42~instances solved by both, 15~require branching: edge branching explores 8~nodes on average (max 61), while Ryan-Foster explores 14~nodes on average (max 99).

\ourparagraph{Dual Stabilization}
Furthermore, we tested dual stabilization to reduce the oscillation of dual values during column generation.
While stabilization reduces the number of column generation iterations by 14\% at the root node (from 126 to 108~on average) and by 18\% overall, this barely affects solving times: the SGM is 6.3~seconds with stabilization versus 6.5~seconds without, and both configurations solve the same 51~instances.

\subsection{Cost of Numerical Safety}\label{sec:numsafety}

\begin{table}[htbp]
    \centering
\small
\begin{tabular}{lrrr}
    \toprule
    Setting & Solved & Time [s] & At Root \\
    \midrule
    \textsc{bnp-full} (safe)   & 51 & 7.6 & 49 \\
    \textsc{bnp-unsafe}        & 51 & \textbf{4.5} & 50 \\
    \bottomrule
\end{tabular}

    \caption{Cost of numerical safety. Solved: instances solved to optimality; Time: shifted geometric mean of solving times (shift of 1s) over instances solved by at least one method; At Root: solved at root node.}
    \label{tab:numerical-safety}
\end{table}

Tab.~\ref{tab:numerical-safety} shows the impact of numerical safety on performance.
The unsafe version \textsc{bnp-unsafe} is approximately 40\% faster than the safe version \textsc{bnp-full} (4.5s vs 7.6s shifted geometric mean) and solves the same number of instances (51): the unsafe version closes t62\_hk48, while the safe version instead closes at62\_p43.
No dual bound failures were detected: on all instances where both versions converge, the safe and unsafe dual bounds agree.
This demonstrates that numerical safety can be ensured with a moderate overhead, preserving the substantial performance gains of the branch-and-price approach, but that it may be beneficial to run the heuristic versions \textsc{heurbnp} and \textsc{heurbnp-rf} in unsafe mode.

\subsection{Branch-Price-and-Cut Performance}\label{sec:bpc}
In order to evaluate the options described in~\autoref{sec:cuts}, we compare five different branch-price-and-cut configurations:
\begin{itemize}
    \item \textsc{bnp-full}: no cutting planes (baseline)
    \item \textsc{withcc}: clique cuts without change to the pricing problem
    \item \textsc{withsrc}: subset row cuts without change to the pricing problem
    \item \textsc{withsrc+cc}: both subset row and clique cuts without change to the pricing problem
    \item \textsc{withsrc+pricing}: subset row cuts with modifications to the pricing problem
\end{itemize}
Tab.~\ref{tab:cuts} reports aggregated results including statistics on the root performance and again reports the ``virtual best'' performance over these methods.

\begin{table}[htbp]
    \centering
\small
\begin{tabular}{lrrrrr}
    \toprule
    Setting & Solved & At Root & Root LP & Time [s] & Ratio \\
    \midrule
    \textsc{bnp-full}          & 51 & 49 & 52 &  7.6 & 1.00 \\
    \textsc{withcc}            & 51 & 50 & 52 &  7.3 & 0.97 \\
    \textsc{withsrc}           & 48 & 48 & 51 &  9.6 & 1.27 \\
    \textsc{withsrc+cc}        & 49 & 49 & 52 &  9.6 & 1.27 \\
    \textsc{withsrc+pricing}   & \textbf{52} & \textbf{51} & 53 & \textbf{7.2} & \textbf{0.96} \\
    \midrule
    Virtual Best               & 52 & --- & --- & 6.0 & 0.80 \\
    \bottomrule
\end{tabular}

    \caption{Branch-price-and-cut performance. Solved: instances solved to optimality; At Root: solved at root node; Root LP: root LP relaxation converged; Time: shifted geometric mean of solving times (shift of 1s); Ratio: time relative to \textsc{bnp-full}.}
    \label{tab:cuts}
\end{table}
First, cutting planes are separated on very few instances: subset-row cuts are added on only 3~of the 84~instances (\texttt{t62\_dantzig42}, \texttt{t84\_eil51} and \texttt{t84\_eil76}, with 11~cuts in a single separation round each), and clique cuts on 2~instances (\texttt{t84\_eil51} and \texttt{t84\_eil76}, with 33 and 26~cuts over 5 and 4~rounds, respectively).
This is a direct consequence of the tight LP relaxation: most instances are solved at the root with little fractionality left to separate.
Adding cutting planes to the RMP without modifying the pricing problem generally fails to help: clique cuts alone (\textsc{withcc}) are roughly neutral, even marginally faster (a 3\% speedup), while solving the same instances, whereas subset row cuts (\textsc{withsrc}) cause a 27\% slowdown and combining both (\textsc{withsrc+cc}) also a 27\% slowdown.
Adding cuts only to the RMP can even be counterproductive: relative to \textsc{bnp-full}, \textsc{withsrc} fails to solve \texttt{t62\_att48}, \texttt{t62\_dantzig42} and \texttt{t84\_eil51} (48~solved), while \textsc{withsrc+cc} fails on \texttt{t62\_dantzig42} and \texttt{t84\_eil51} (49~solved).
The number of column generation iterations remains nearly identical across configurations on instances solved by all methods, so the slowdown is attributable to the overhead of cut separation and LP re-optimization, which is not compensated by tighter bounds.

Respecting the subset row cuts in the pricing problem (\textsc{withsrc+pricing}) reverses this effect on \texttt{t84\_eil51}: the cuts close the root node, so the instance is solved without branching (in 13.7~s, versus the 25.5~s and 15~branch-and-bound nodes needed by \textsc{bnp-full}), recovering the instance that \textsc{withsrc} loses.
By generating columns that already satisfy the cuts, this configuration solves 52~instances, one more than \textsc{bnp-full} (it additionally closes \texttt{t62\_hk48}), with a 4\% lower shifted geometric mean time (7.2~versus 7.6~seconds), the best of all cutting-plane configurations.
The virtual best over all cutting plane configurations solves 52~instances, the same as \textsc{withsrc+pricing}.

\section{Conclusion}
To summarize, our computational results show that a branch-price-and-cut approach significantly improves over the previous state-of-the-art in solving the length-constrained cycle partition problem, both in terms of primal solution quality and dual performance.
The best-performing branch-price-and-cut algorithm (\textsc{withsrc+pricing}) solves 52~instances from the standard benchmark set, compared to 40~for the previous branch-and-cut method, and scales to instances with up to 76~nodes (vs. 52 previously reported).
The main driver of this performance improvement is the set partitioning formulation: it provides substantially tighter LP relaxations, with an average root gap of 0.4\% compared to 30.3\% for the subtour elimination formulation.

As a result, most instances are solved at the root node without branching.
A natural question is whether the set-partitioning formulation satisfies the \emph{round-up property}, i.e., whether the optimal integer objective always equals the LP relaxation value rounded up, $z^* = \lceil z_{LP} \rceil$.
Our results show that this does not hold in general: on \texttt{t84\_eil51} the LP value rounds up to~8 while the integer optimum is~9, which is exactly why branching is required there.
On the other 50 of the 52~instances whose root LP converges the property does hold, and on the single open instance \texttt{t84\_eil76} it cannot be decided, as its LP rounds up to~8 while the best known solution uses 9~cycles.
This near-universal round-up behavior is nonetheless what allows most instances to be solved at the root once a matching primal solution is found.
For a study of the conditions under which the round-up property holds, see~\cite{kilian}.

Among the algorithmic components, symmetry breaking, bidirectional labeling, and improved initial upper bounds contribute the most.
The greedy pricing heuristic succeeds in 93\% of pricing calls and produces root LP bounds matching the exact algorithm, while avoiding the more expensive labeling procedure.
Subset row cuts respected in pricing are the only cutting plane configuration that improves performance, reducing the solving time without losing any instances.
When branching is needed, edge branching proves more robust than Ryan-Foster branching, producing smaller search trees.
Interestingly, techniques that reduce the number of column generation iterations, namely dual stabilization and triangle inequality exploitation, do not improve solving times, suggesting that the pricing problem is solved efficiently enough that fewer iterations do not compensate for the added overhead.

All results are based on a numerically safe combination of floating-point and integer arithmetic.
While we did not notice any incorrect results when using unsafe floating-point computation, we could confirm that numerical safe computation incurs only a moderate overhead and preserves the substantial performance gains of the branch-price-and-cut approach.

The most pressing issue for future work is the efficiency of the pricing routine: instances beyond 76~nodes remain unsolved, and pricing dominates the running time.
Closing these instances will likely require stronger dual bounds or a more scalable pricing algorithm.
Recent advances in accelerating dynamic-programming-based pricing, such as parallel labeling for the resource-constrained shortest path problem~\cite{petersen2025parallel} or GPU-accelerated state expansion~\cite{tardivo2026gpu}, offer a promising avenue.{}
The potential of further improvements, such as parallelizing the remaining sequential parts of the branch-and-bound search or separating additional cutting planes, can only be unlocked once pricing is no longer the bottleneck.

\section*{Funding}
Research reported in this paper was partially supported through the Research Campus MODAL funded by the German Federal Ministry of Education and Research (fund numbers 05M14ZAM, 05M20ZBM) and the Deutsche Forschungsgemeinschaft (DFG) through the DFG Cluster of Excellence MATH+.

\section*{Data availability}
The benchmark instances and the scripts used to generate all reported results are openly available in the \texttt{lccp-ejor} repository at \url{https://github.com/mmghannam/lccp-ejor}.

\appendix
\begin{landscape}
\section{Per-Instance Results}\label{sec:appendix}
\begin{longtable}{l r r | r r r r r r c | r r r r r c}
\caption{Per-instance results comparing \textsc{bnp-full} (branch-price-and-cut) with \textsc{bnc-sec} (branch-and-cut with subtour elimination).
BKS: best known solution; PB: primal bound; DB: dual bound; DBr: dual bound at root; Gap: optimality gap (\%); Time: solving time in seconds; Nodes: branch-and-bound nodes; Status: opt (optimal) or TL (time limit).}
\label{tab:supplementary}\\
\toprule
& & & \multicolumn{7}{c|}{\textsc{bnp-full}} & \multicolumn{6}{c}{\textsc{bnc-sec}} \\
Instance & $n$ & BKS & PB & DBr & DB & Gap & Time & Nodes & St. & PB & DB & Gap & Time & Nodes & St. \\
\midrule
\endfirsthead
\multicolumn{16}{c}{\tablename\ \thetable{} -- continued from previous page} \\
\toprule
& & & \multicolumn{7}{c|}{\textsc{bnp-full}} & \multicolumn{6}{c}{\textsc{bnc-sec}} \\
Instance & $n$ & BKS & PB & DBr & DB & Gap & Time & Nodes & St. & PB & DB & Gap & Time & Nodes & St. \\
\midrule
\endhead
\midrule
\multicolumn{16}{r}{Continued on next page} \\
\endfoot
\bottomrule
\endlastfoot
\texttt{t62\_burma14} & 14 & 5 & 5 & 5 & 5 & 0 & 0.1 & 1 & opt & 5 & 5 & 0 & 0.1 & 43 & opt \\
\texttt{t84\_burma14} & 14 & 6 & 6 & 6 & 6 & 0 & 0.1 & 1 & opt & 6 & 6 & 0 & 0.0 & 1 & opt \\
\texttt{t62\_ulysses16} & 16 & 4 & 4 & 4 & 4 & 0 & 0.1 & 1 & opt & 4 & 4 & 0 & 0.0 & 1 & opt \\
\texttt{t84\_ulysses16} & 16 & 6 & 6 & 6 & 6 & 0 & 0.1 & 1 & opt & 6 & 6 & 0 & 0.1 & 1 & opt \\
\texttt{at62\_br17} & 17 & 5 & 5 & 5 & 5 & 0 & 0.1 & 1 & opt & 5 & 5 & 0 & 0.1 & 1 & opt \\
\texttt{at84\_br17} & 17 & 6 & 6 & 6 & 6 & 0 & 0.0 & 1 & opt & 6 & 6 & 0 & 0.0 & 1 & opt \\
\texttt{t62\_gr17} & 17 & 5 & 5 & 5 & 5 & 0 & 0.1 & 1 & opt & 5 & 5 & 0 & 0.2 & 1 & opt \\
\texttt{t84\_gr17} & 17 & 8 & 8 & 8 & 8 & 0 & 0.1 & 1 & opt & 8 & 8 & 0 & 0.0 & 1 & opt \\
\texttt{t62\_gr21} & 21 & 5 & 5 & 5 & 5 & 0 & 0.1 & 1 & opt & 5 & 5 & 0 & 0.8 & 34 & opt \\
\texttt{t84\_gr21} & 21 & 8 & 8 & 8 & 8 & 0 & 0.1 & 1 & opt & 8 & 8 & 0 & 0.5 & 21 & opt \\
\texttt{t62\_ulysses22} & 22 & 5 & 5 & 5 & 5 & 0 & 0.3 & 1 & opt & 5 & 5 & 0 & 0.1 & 1 & opt \\
\texttt{t84\_ulysses22} & 22 & 7 & 7 & 7 & 7 & 0 & 0.1 & 1 & opt & 7 & 7 & 0 & 0.7 & 67 & opt \\
\texttt{t62\_gr24} & 24 & 5 & 5 & 5 & 5 & 0 & 0.5 & 1 & opt & 5 & 5 & 0 & 1.1 & 25 & opt \\
\texttt{t84\_gr24} & 24 & 7 & 7 & 7 & 7 & 0 & 0.1 & 1 & opt & 7 & 7 & 0 & 0.2 & 1 & opt \\
\texttt{t62\_fri26} & 26 & 6 & 6 & 6 & 6 & 0 & 1.0 & 1 & opt & 6 & 6 & 0 & 1.9 & 158 & opt \\
\texttt{t84\_fri26} & 26 & 8 & 8 & 8 & 8 & 0 & 0.1 & 1 & opt & 8 & 8 & 0 & 1.1 & 154 & opt \\
\texttt{t62\_bayg29} & 29 & 5 & 5 & 5 & 5 & 0 & 1.1 & 1 & opt & 5 & 5 & 0 & 12.7 & 8494 & opt \\
\texttt{t62\_bays29} & 29 & 6 & 6 & 6 & 6 & 0 & 0.6 & 1 & opt & 6 & 6 & 0 & 3.4 & 814 & opt \\
\texttt{t84\_bayg29} & 29 & 8 & 8 & 8 & 8 & 0 & 0.1 & 1 & opt & 8 & 8 & 0 & 17.1 & 5969 & opt \\
\texttt{t84\_bays29} & 29 & 8 & 8 & 8 & 8 & 0 & 0.1 & 1 & opt & 8 & 8 & 0 & 7.4 & 1018 & opt \\
\texttt{at62\_ftv33} & 33 & 7 & 7 & 7 & 7 & 0 & 1.2 & 1 & opt & 7 & 7 & 0 & 3113.6 & 3353943 & opt \\
\texttt{at84\_ftv33} & 33 & 9 & 9 & 9 & 9 & 0 & 0.1 & 1 & opt & 9 & 9 & 0 & 13.0 & 403 & opt \\
\texttt{at62\_ftv35} & 35 & 6 & 6 & 6 & 6 & 0 & 3.8 & 1 & opt & 6 & 6 & 0 & 3285.3 & 14828814 & opt \\
\texttt{at84\_ftv35} & 35 & 9 & 9 & 9 & 9 & 0 & 0.3 & 1 & opt & 9 & 9 & 0 & 165.6 & 26250 & opt \\
\texttt{at62\_ftv38} & 38 & 6 & 6 & 6 & 6 & 0 & 9.6 & 1 & opt & 6 & 5 & 16.7 & TL & 36061643 & TL \\
\texttt{at84\_ftv38} & 38 & 9 & 9 & 9 & 9 & 0 & 0.7 & 1 & opt & 9 & 9 & 0 & 546.9 & 278079 & opt \\
\texttt{t62\_dantzig42} & 42 & 6 & 6 & 6 & 6 & 0 & 113.4 & 41 & opt & 6 & 6 & 0 & 1024.9 & 2590519 & opt \\
\texttt{t62\_swiss42} & 42 & 6 & 6 & 6 & 6 & 0 & 87.5 & 1 & opt & 6 & 6 & 0 & 1630.6 & 3776539 & opt \\
\texttt{t84\_dantzig42} & 42 & 9 & 9 & 9 & 9 & 0 & 0.9 & 1 & opt & 9 & 9 & 0 & 1364.0 & 253146 & opt \\
\texttt{t84\_swiss42} & 42 & 9 & 9 & 9 & 9 & 0 & 0.9 & 1 & opt & 9 & 9 & 0 & 233.1 & 53847 & opt \\
\texttt{at62\_p43} & 43 & 1 & 1 & 1 & 1 & 0 & 297.3 & 1 & opt & 1 & 1 & 0 & 0.1 & 0 & opt \\
\texttt{at84\_p43} & 43 & 2 & 2 & 2 & 2 & 0 & 3.0 & 1 & opt & 2 & 2 & 0 & 0.1 & 1 & opt \\
\texttt{at62\_ftv44} & 44 & 6 & 6 & 6 & 6 & 0 & 77.4 & 1 & opt & 7 & 5 & 28.6 & TL & 10118677 & TL \\
\texttt{at84\_ftv44} & 44 & 9 & 9 & 9 & 9 & 0 & 0.7 & 1 & opt & 9 & 9 & 0 & 6617.2 & 1007542 & opt \\
\texttt{at62\_ftv47} & 47 & 6 & 6 & 6 & 6 & 0 & 215.6 & 1 & opt & 7 & 6 & 14.3 & TL & 6819795 & TL \\
\texttt{at84\_ftv47} & 47 & 8 & 8 & 8 & 8 & 0 & 1.2 & 1 & opt & 8 & 8 & 0 & 5322.9 & 6130397 & opt \\
\texttt{at62\_ry48p} & 48 & 6 & - & - & - & 100.0 & TL & - & TL & 6 & 5 & 16.7 & TL & 11727846 & TL \\
\texttt{at84\_ry48p} & 48 & 8 & 8 & 8 & 8 & 0 & 7.6 & 1 & opt & 9 & 8 & 11.1 & TL & 6659592 & TL \\
\texttt{t62\_att48} & 48 & 6 & 6 & 6 & 6 & 0 & 4599.6 & 1 & opt & 6 & 6 & 0 & 1096.1 & 2278842 & opt \\
\texttt{t62\_gr48} & 48 & 6 & 6 & 6 & 6 & 0 & 98.2 & 1 & opt & 7 & 5 & 28.6 & TL & 14151852 & TL \\
\texttt{t62\_hk48} & 48 & 6 & - & - & - & 100.0 & TL & - & TL & 6 & 6 & 0 & 6026.9 & 20699967 & opt \\
\texttt{t84\_att48} & 48 & 8 & 8 & 8 & 8 & 0 & 26.7 & 1 & opt & 8 & 8 & 0 & 1010.4 & 186278 & opt \\
\texttt{t84\_gr48} & 48 & 8 & 8 & 8 & 8 & 0 & 5.6 & 1 & opt & 9 & 8 & 11.1 & TL & 9427085 & TL \\
\texttt{t84\_hk48} & 48 & 9 & 9 & 9 & 9 & 0 & 3.6 & 1 & opt & 9 & 9 & 0 & 1904.9 & 540077 & opt \\
\texttt{t62\_eil51} & 51 & 6 & 6 & 6 & 6 & 0 & 947.2 & 1 & opt & 7 & 5 & 28.6 & TL & 13281898 & TL \\
\texttt{t84\_eil51} & 51 & 9 & 9 & 8 & 9 & 0 & 25.5 & 15 & opt & 9 & 8 & 11.1 & TL & 9935924 & TL \\
\texttt{t62\_berlin52} & 52 & 6 & - & - & - & 100.0 & TL & - & TL & 7 & 6 & 14.3 & TL & 14358987 & TL \\
\texttt{t84\_berlin52} & 52 & 9 & - & - & - & 100.0 & TL & - & TL & 9 & 8 & 11.1 & TL & 5048215 & TL \\
\texttt{at62\_ft53} & 53 & 5 & - & - & - & 100.0 & TL & - & TL & 7 & 4 & 42.9 & TL & 8070726 & TL \\
\texttt{at84\_ft53} & 53 & 8 & 8 & 8 & 8 & 0 & 47.4 & 1 & opt & 10 & 6 & 40.0 & TL & 2166224 & TL \\
\texttt{at62\_ftv55} & 55 & 6 & 6 & 6 & 6 & 0 & 448.8 & 1 & opt & 7 & 5 & 28.6 & TL & 7753333 & TL \\
\texttt{at84\_ftv55} & 55 & 9 & 9 & 9 & 9 & 0 & 9.2 & 1 & opt & 11 & 7 & 36.4 & TL & 648741 & TL \\
\texttt{t62\_brazil58} & 58 & 5 & - & - & - & 100.0 & TL & - & TL & 6 & 5 & 16.7 & TL & 15534631 & TL \\
\texttt{t84\_brazil58} & 58 & 8 & - & - & - & 100.0 & TL & - & TL & 8 & 7 & 12.5 & TL & 8560266 & TL \\
\texttt{at62\_ftv64} & 64 & 7 & - & - & - & 100.0 & TL & - & TL & 8 & 5 & 37.5 & TL & 2958768 & TL \\
\texttt{at84\_ftv64} & 64 & 10 & 10 & 10 & 10 & 0 & 26.2 & 1 & opt & 11 & 7 & 36.4 & TL & 358873 & TL \\
\texttt{at62\_ft70} & 70 & 6 & - & - & - & 100.0 & TL & - & TL & 6 & 4 & 33.3 & TL & 4094939 & TL \\
\texttt{at62\_ftv70} & 70 & 9 & - & - & - & 100.0 & TL & - & TL & 9 & 5 & 44.4 & TL & 2445597 & TL \\
\texttt{at84\_ft70} & 70 & 9 & - & - & - & 100.0 & TL & - & TL & 9 & 6 & 33.3 & TL & 1727481 & TL \\
\texttt{at84\_ftv70} & 70 & 9 & 9 & 9 & 9 & 0 & 177.6 & 1 & opt & 11 & 7 & 36.4 & TL & 872838 & TL \\
\texttt{t62\_st70} & 70 & 6 & - & - & - & 100.0 & TL & - & TL & 7 & 5 & 28.6 & TL & 6432455 & TL \\
\texttt{t84\_st70} & 70 & 9 & 9 & 9 & 9 & 0 & 76.2 & 1 & opt & 9 & 7 & 22.2 & TL & 401218 & TL \\
\texttt{t62\_eil76} & 76 & 6 & - & - & - & 100.0 & TL & - & TL & 7 & 5 & 28.6 & TL & 3945851 & TL \\
\texttt{t62\_pr76} & 76 & 6 & - & - & - & 100.0 & TL & - & TL & 7 & 5 & 28.6 & TL & 3701310 & TL \\
\texttt{t84\_eil76} & 76 & 9 & 9 & 8 & 8 & 11.1 & TL & 5 & TL & 10 & 7 & 30.0 & TL & 730776 & TL \\
\texttt{t84\_pr76} & 76 & 9 & 9 & 9 & 9 & 0 & 1412.8 & 1 & opt & 10 & 7 & 30.0 & TL & 137363 & TL \\
\texttt{t62\_gr96} & 96 & 7 & - & - & - & 100.0 & TL & - & TL & 7 & 5 & 28.6 & TL & 1917562 & TL \\
\texttt{t84\_gr96} & 96 & 9 & - & - & - & 100.0 & TL & - & TL & 10 & 6 & 40.0 & TL & 59244 & TL \\
\texttt{t62\_rat99} & 99 & 7 & - & - & - & 100.0 & TL & - & TL & 7 & 5 & 28.6 & TL & 1430164 & TL \\
\texttt{t84\_rat99} & 99 & 10 & - & - & - & 100.0 & TL & - & TL & 10 & 7 & 30.0 & TL & 146129 & TL \\
\texttt{t62\_kroA100} & 100 & 8 & - & - & - & 100.0 & TL & - & TL & 8 & 5 & 37.5 & TL & 426410 & TL \\
\texttt{t62\_kroB100} & 100 & 8 & - & - & - & 100.0 & TL & - & TL & 8 & 5 & 37.5 & TL & 618315 & TL \\
\texttt{t62\_kroC100} & 100 & 7 & - & - & - & 100.0 & TL & - & TL & 7 & 5 & 28.6 & TL & 461929 & TL \\
\texttt{t62\_kroD100} & 100 & 7 & - & - & - & 100.0 & TL & - & TL & 7 & 5 & 28.6 & TL & 1541628 & TL \\
\texttt{t62\_kroE100} & 100 & 7 & - & - & - & 100.0 & TL & - & TL & 7 & 5 & 28.6 & TL & 1361385 & TL \\
\texttt{t62\_rd100} & 100 & 7 & - & - & - & 100.0 & TL & - & TL & 7 & 5 & 28.6 & TL & 369576 & TL \\
\texttt{t84\_kroA100} & 100 & 9 & - & - & - & 100.0 & TL & - & TL & 10 & 6 & 40.0 & TL & 16871 & TL \\
\texttt{t84\_kroB100} & 100 & 9 & - & - & - & 100.0 & TL & - & TL & 10 & 6 & 40.0 & TL & 4821 & TL \\
\texttt{t84\_kroC100} & 100 & 9 & - & - & - & 100.0 & TL & - & TL & 11 & 6 & 45.5 & TL & 6357 & TL \\
\texttt{t84\_kroD100} & 100 & 9 & - & - & - & 100.0 & TL & - & TL & 10 & 7 & 30.0 & TL & 24214 & TL \\
\texttt{t84\_kroE100} & 100 & 9 & - & - & - & 100.0 & TL & - & TL & 11 & 6 & 45.5 & TL & 17521 & TL \\
\texttt{t84\_rd100} & 100 & 10 & - & - & - & 100.0 & TL & - & TL & 11 & 6 & 45.5 & TL & 4416 & TL \\
\texttt{at62\_kro124p} & 124 & 7 & - & - & - & 100.0 & TL & - & TL & 7 & 4 & 42.9 & TL & 202101 & TL \\
\texttt{at84\_kro124p} & 124 & 9 & - & - & - & 100.0 & TL & - & TL & 10 & 6 & 40.0 & TL & 71479 & TL \\
\end{longtable}

\end{landscape}
\bibliographystyle{plainnat}
\bibliography{bibliography}
\end{document}